\documentclass[11pt,letterpaper]{article}

\usepackage[T1]{fontenc}
\usepackage[utf8]{inputenc}
\usepackage[letterpaper,margin=1in]{geometry}

\usepackage{amsmath}
\usepackage{amssymb}
\usepackage{amsthm}
\usepackage{mathtools}
\usepackage{microtype}

\usepackage{authblk}

\usepackage{natbib}
\renewcommand{\citet}[1]{\cite{#1}}
\renewcommand{\citep}[1]{\cite{#1}}
\usepackage[left,mathlines]{lineno}

\usepackage[colorlinks=true,citecolor=blue,linkcolor=blue,urlcolor=blue]{hyperref}
\usepackage[capitalize,noabbrev]{cleveref}
\newcommand{\Pb}{\bbP}
\newcommand{\dif}{\mathrm{d}}

\newcommand{\borel}{\mathfrak{B}}

\newcommand{\lrb}[1]{\left(#1\right)}
\newcommand{\lsb}[1]{\left[#1\right]}
\newcommand{\lcb}[1]{\left\{#1\right\}}
\newcommand{\labs}[1]{\left\lvert#1\right\rvert}
\newcommand{\norm}[1]{\lVert #1\rVert}

\DeclareMathOperator*{\argmin}{argmin}
\DeclareMathOperator{\law}{\mathsf{Law}}

\newcommand{\bbE}{\mathbb{E}}
\newcommand{\bbI}{\mathbb{I}}
\newcommand{\bbN}{\mathbb{N}}
\newcommand{\bbP}{\mathbb{P}}
\newcommand{\bbR}{\mathbb{R}}

\newcommand{\cA}{\mathcal{A}}
\newcommand{\cF}{\mathcal{F}}
\newcommand{\cI}{\mathcal{I}}

\newcommand{\cR}{\mathcal{R}}
\newcommand{\cS}{\mathcal{S}}
\newcommand{\cX}{\mathcal{X}}

\theoremstyle{plain}
\newtheorem{theorem}{Theorem}
\newtheorem{lemma}[theorem]{Lemma}
\newtheorem{proposition}[theorem]{Proposition}
\newtheorem{corollary}[theorem]{Corollary}
\newtheorem{claim}[theorem]{Claim}

\theoremstyle{definition}
\newtheorem{assumption}{Assumption}
\newtheorem{definition}[theorem]{Definition}

\theoremstyle{remark}

\title{New Bounds for the Last Iterate of the\\Stochastic subGradient Method}
\author[1]{Guglielmo Beretta}
\author[2]{Tommaso Cesari}
\author[3]{Roberto Colomboni}
\author[1]{Andrea Paudice}

\affil[1]{Department of Computer Science, Aarhus University, Aarhus, Denmark}
\affil[2]{School of Electrical Engineering and Computer Science, University of Ottawa, Ottawa, Canada}
\affil[3]{School of Mathematics, University of Bristol, Bristol, United Kingdom}

\affil[ ]{
\texttt{gberetta@cs.au.dk},
\texttt{tcesari@uottawa.ca},
\texttt{roberto.colomboni@bristol.ac.uk},
\texttt{apaudice@cs.au.dk}
}
\date{}

\begin{document}

\maketitle
%
\begin{abstract}
\noindent We study the last iterate of the stochastic subgradient method for one-dimensional convex Lipschitz objectives. For a fixed horizon $n$, we consider the standard fixed stepsizes $\eta =\Theta(1/\sqrt n)$. We prove that, for such stepsize policies, under additive i.i.d.\ subgradient noise with uniformly bounded variance, the last iterate features an optimization error of order $1/\sqrt n$, thereby removing the extra $(\log n)$ factor present in existing generic bounds. On the other hand, we show that without the i.i.d.\ assumption, the optimization error can be of order $(\log n)/\sqrt n$. Thus, under the uniformly bounded variance assumption alone, the last iterate of SsGM is suboptimal even in dimension one, resolving negatively an open problem posed in \cite{koren-segal2020}.
\end{abstract}
\section{Introduction}
The (projected) \emph{Stochastic subGradient Method} (SsGM) \citet{Ermoliev1969} is arguably one of the most popular algorithms in constrained stochastic optimization. Its success is largely due to its efficiency and provable performance. Classical analyses for convex Lipschitz objectives, which are also the focus of this work, typically control the expected optimization error of the \emph{average iterate} \cite{Nemirovski1983,Nemirovski2009,Liu2023b,Bach2024}. In the finite-horizon setting, where the number of iterations $n$ is known in advance, one may use a \emph{standard fixed stepsize} $(\eta_i)_{1\leq i\leq n}\equiv \eta$ with $\eta=\Theta(1/\sqrt n)$, and return the average iterate\footnote{Throughout the paper, this notation has its usual finite-horizon meaning: there are constants $0<\underline c\le \overline c<\infty$ and $n_0\in\bbN$ such that, for every $n\ge n_0$, all iterations use the same stepsize $\eta=c_n/\sqrt n$ with $c_n\in[\underline c,\overline c]$. Constants hidden in asymptotic notation may depend on $\underline c$ and $\overline c$, but not on $n$.}. Assuming uniformly bounded variance of the stochastic subgradients and finite diameter of the feasible set, averaging yields optimization error of order $1/\sqrt n$, which is optimal for this class of problems \cite{Nemirovski1983,Agarwal2012}.

\noindent More recently, attention has shifted to the last iterate of the stochastic subgradient method \cite{Shamir2013,Harvey2019a,Jain2021,Liu2024,Eldowa2024}. This is the iterate naturally returned in practice, due to its simplicity and often superior empirical performance. However, last-iterate guarantees are substantially more delicate. The best known upper bound for the standard fixed-stepsize choice $\eta=\Theta(1/\sqrt n)$ is of order $(\log n)/\sqrt n$, leaving an extra logarithmic factor compared with the optimal rate achieved by averaging \cite{Shamir2013}. This logarithmic loss is known to be unavoidable even in deterministic setting if the dimension is allowed to scale with the horizon, for instance when $d \approx n$ \cite{zamani-glineur2025}. However, that does not settle the problem in fixed dimension. In fact, this is precisely the question asked in Open Problem~1 of \cite{koren-segal2020}.

\paragraph{Contributions.} In this work, we answer this question negatively. More precisely, for every choice of constants $0<\underline c\le\overline c<\infty$, for every sufficiently large horizon $n$, and for every fixed stepsize $\eta=c_n/\sqrt n$ with $c_n\in[\underline c,\overline c]$, we construct a one-dimensional stochastic convex optimization instance on which the last iterate suffers expected optimization error of order $(\log n)/\sqrt n$ (\Cref{thm:clock-counterexamples}). Thus, averaging is not just a proof device: without additional structure, the last iterate is genuinely suboptimal even in dimension one. On the positive side, still for $d=1$, we show that if the additive subgradient noise is both state-independent and time-homogeneous, then every standard fixed-stepsize family $\eta=\Theta(1/\sqrt n)$ yields the optimal last-iterate order $1/\sqrt n$, with an explicit constant depending only on the problem parameters and on the hidden constants in the $\Theta$ notation (\Cref{thm:bound_fh_simpler}).

\section{Related work}

\paragraph{Existing last-iterate bounds.}
Besides the standard constant stepsize used in the finite-horizon setting, several alternative stepsize policies have been proposed to improve the behavior of the last iterate \cite{Jain2021,Liu2024}. \cite{Jain2021} were the first to introduce a non-standard finite-horizon policy under which the last iterate recovers the optimal $1/\sqrt n$ rate. \cite{Liu2024} later proposed a different finite-horizon policy with the same optimal guarantee. In the anytime setting, where the horizon is not fixed in advance, the standard choice is a decreasing stepsize. \cite{Shamir2013} showed that, for $\eta_i=\eta/\sqrt i$, the last iterate has optimization error of order at most $(\log n)/\sqrt n$, and \cite{Harvey2019a,Harvey2024} showed that this logarithmic factor is unavoidable even in the deterministic setting. Improved anytime guarantees can be obtained with more refined stepsize policies (e.g., see \cite{Parletta2025}), but an extra $\operatorname{polylog}(n)$ factor is necessary in infinite-dimensional, or sufficiently high-dimensional, settings \cite{Kornowski2026}.

\paragraph{Fixed-dimensional setting.}
The fixed-dimensional case is much less understood. \cite{koren-segal2020} supported the possibility of removing the logarithmic factor by showing that, in dimension $d=1$, for the absolute-value objective and a specific stochastic oracle model, the algorithm dynamics reduce to a biased one-sided random walk and the optimal $1/\sqrt n$ rate is attained. \cite{liu-lu2021} extended this positive result to a class of one-dimensional piecewise-linear convex functions under sub-Gaussian noise. They also complemented their upper bounds with a lower bound of order $(\log d)/\sqrt n$, which, however, does not rule out the optimal rate when $d$ is fixed. In comparison, our positive result applies to arbitrary one-dimensional convex Lipschitz objectives, under a state-independent time-homogeneous additive noise model with finite second moment, and it is uniform over the stepsize family $\eta=\Theta(1/\sqrt n)$. On the other hand, our lower bound shows that, under the sole uniformly bounded variance assumption, the logarithmic factor is unavoidable even in $d=1$ for the same stepsize family.

\section{Problem Setting}\label{sec:problem_setting}
\paragraph{Notation.} Throughout, $k$ and $n$ will always denote natural numbers. For a fixed fuction $f$ and $x \in \cX$, let us define $\Delta(x) \coloneq f(x)-f_{\star}$ and $\operatorname{dist}(x, \cX_{\star}) \coloneq \inf_{z \in \cX_{\star}} |x-z|$. Given a closed convex set $\cS \subseteq \bbR$, we write $\Pi_\cS(x) \coloneq \argmin_{z \in \cS} \labs{x -z}$ for the orthogonal projection of $x \in \bbR$ onto $\cS$.
We introduce an abuse of notation for closed intervals: given $-\infty \le a \le b \le +\infty$ such that $(a,b)\notin\{(-\infty,-\infty),(+\infty,+\infty)\}$, we write
\begin{equation*}
[a, b] \coloneq \{ x \in \bbR : a \le x \le b\},
\end{equation*}
which is a subset of $\bbR$ even if $a$ and/or $b$ are infinite.
With this notation, if $\cS \coloneq [a, b]$ then $\Pi_\cS(x) = \max\{ a, \min\{x, b \}\}$ for every $x \in \bbR$.
Given a convex $\cS \subseteq \bbR$, a point $x \in \cS$ and a convex function $h \colon \cS \to \bbR$, we call $g\in \bbR$ a \emph{relative subgradient of $h$ at $x$ with respect to $\cS$} and write $g \in \partial_{\cS} h(x)$ if
$h(y) \ge h(x) + g \cdot (y - x)$ for every $y \in \cS$.
For a convex function $h$ on an interval, we write $D_{-}h$ and $D_{+}h$ for its left and right derivatives, whenever they are defined. We use the standard fact that, if $g\in\partial_{\cS}h(x)$, then $D_{-}h(x)\le g$ whenever $D_{-}h(x)$ exists, and $g\le D_{+}h(x)$ whenever $D_{+}h(x)$ exists. For all $x\in\bbR$, we denote by $\delta_x$ the Dirac probability measure at $x$.

\paragraph{Problem Formulation.} We consider the following optimization problem:
\begin{equation*}
\underset{x \in \cX}{\text{minimize}} \quad f(x)
\end{equation*}
under the following assumptions. 
\begin{assumption}[Feasibility Domain]
\label{ass:feasibility}
$\cX \subseteq \bbR$ is closed and convex, and $\cX_{\star} \coloneq \arg \min_{x \in \cX} f(x) \neq \emptyset$. Moreover, $f_{\star} \coloneq \min_{x \in \cX} f(x) \in \bbR$.
\end{assumption}
\begin{assumption}[Convexity and Lipschitzness]
\label{ass:convexity_lipschitzness}
$f\colon \cX \rightarrow \bbR$ is convex and $L$-Lipschitz continuous on $\cX$. 
\end{assumption}
\noindent 
Assumption~\ref{ass:convexity_lipschitzness} implies that all secant slopes of $f$ over $\cX$ belong to $[-L,L]$, and consequently,%
\footnote{
We observe that $\partial_{\cX}f(x)$ is an unbounded interval at every $x$ that is a boundary point of $\cX$, and this motivates the explicit restriction $\labs{G_k} \le L$ given in Definition~\ref{def:sgd}.}
that $\partial_{\cX}f(x)\cap[-L,L]\neq\emptyset$ for every $x \in \cX$.
Under Assumptions~\ref{ass:feasibility}--\ref{ass:convexity_lipschitzness}, 
we will consider the following formulation of SsGM.
For every trajectory, define the history
\begin{equation*}
    \cF_k
\coloneq
    \sigma(X_1,W_1,\dots,W_{k-1}),
\qquad
    k\ge1,
\end{equation*}
with the convention that $\cF_1\coloneq\sigma(X_1)$.
\begin{definition}[Projected Stochastic (sub)Gradient Method]\label{def:sgd}
Given an $\cX$-valued random variable $X_1$, a sequence of random variables $(W_k)_{k \in \bbN}$, a sequence of random variables $(G_k)_{k \in \bbN}$, where $G_k \in \partial_{\cX} f(X_k) \cap [-L,L]$ is $\cF_k$-measurable, and given $\eta >0 $, the stochastic (sub)gradient descent defines the iterates
\begin{equation}\label{eq:sgd}
\forall k \geq 1, \quad X_{k+1} \coloneq \Pi_{\cX}(X_k - \eta \cdot (G_k + W_k)),
\end{equation}
with $\Pi_\cX(x) \coloneq \argmin_{z \in \cX} \labs{x -z}$ being the orthogonal projection of $x \in \bbR$ onto $\cX$.
\end{definition}

\noindent The term $W_k$ can be regarded as an additive noise term that alters the value of $G_k$ when updating $X_k$.
The following condition is the additive-noise analogue of the conditional-unbiasedness and conditional-second-moment oracle condition described in \citep{koren-segal2020}.
The positive result strengthens this moment condition to a kernel-oracle specification and imposes both structural symmetries below. The lower bounds show that, even within this structured kernel class, imposing only one of the two symmetries is not enough to remove the logarithm.
\begin{assumption}[Variance Condition]\label{ass:centered_additive_l2_noise}
There exists $\sigma > 0$ such that
\begin{equation*}\label{eq:centered-l2-noise}
\forall k \ge 1, \qquad \bbE[W_k\mid \cF_k]=0,
\qquad
\bbE[W_k^2\mid \cF_k]\le \sigma^2.
\end{equation*}
\end{assumption}
\noindent To separate the two structural properties that appear in our positive result, we strengthen the above into the following kernel-oracle formulation. %
\begin{assumption}[Centered additive kernel oracle]
\label{ass:additive-kernel-oracle}
There exists $\sigma>0$ and a specified family of probability kernels
\[
Q_k\colon\cX\times\borel(\bbR)\to[0,1],
\qquad k\ge1,
\]
such that, for every $k\ge1$, every $x\in\cX$, and every $A\in\borel(\bbR)$,
\begin{equation*}
\int_{\bbR} w\,Q_k(x,\dif w)=0,
\qquad
\int_{\bbR} w^2\,Q_k(x,\dif w)\le \sigma^2,
\end{equation*}
and
\begin{equation*}
\Pb(W_k\in A\mid\cF_k)=Q_k(X_k,A)
\qquad\text{a.s.}
\end{equation*}
\end{assumption}
\noindent The family of kernels is part of the oracle specification, not just an a posteriori representation along the realized trajectory: $Q_k(x,\cdot)$ is the noise law that the oracle would use at time $k$ if the queried state were $x$. We can show that
Assumption~\ref{ass:additive-kernel-oracle} implies Assumption~\ref{ass:centered_additive_l2_noise}: indeed, since $X_k$ is $\cF_k$-measurable,
\[
\bbE[W_k\mid \cF_k]
=
\int_{\bbR} w\,Q_k(X_k,\dif w)
=
0,
\]
and
\[
\bbE[W_k^2\mid \cF_k]
=
\int_{\bbR} w^2\,Q_k(X_k,\dif w)
\le
\sigma^2.
\]
Thus Assumption~\ref{ass:additive-kernel-oracle} is not a weaker moment condition, but a more structured oracle specification: it describes the conditional law that would be used at every time and every queried state.
\begin{assumption}[Time-homogeneous additive kernel]
\label{ass:time-homogeneous-noise}
The kernels in Assumption~\ref{ass:additive-kernel-oracle} do not depend on time: there exists a probability kernel $Q\colon\cX\times\borel(\bbR)\to[0,1]$ such that
\[
Q_k(x,\cdot)=Q(x,\cdot)
\qquad \forall k\ge1,\ \forall x\in\cX.
\]
\end{assumption}
\begin{assumption}[State-independent additive kernel]
\label{ass:state-independent-noise}
The kernels in Assumption~\ref{ass:additive-kernel-oracle} do not depend on the queried state: for every $k\ge1$ there exists a probability measure $\mu_k$ on $\bbR$ such that
\[
Q_k(x,\cdot)=\mu_k(\cdot)
\qquad \forall x\in\cX.
\]
\end{assumption}
\noindent If Assumptions~\ref{ass:additive-kernel-oracle}, \ref{ass:time-homogeneous-noise}, and \ref{ass:state-independent-noise} all hold, then there exists a single centered law $\mu$ with second moment at most $\sigma^2$ such that
\[
Q_k(x,\cdot)=\mu(\cdot)
\qquad \forall k\ge1,\ \forall x\in\cX.
\]
Consequently, for every predictable SsGM trajectory, the realized noise sequence $(W_k)_{k \in \bbN}$ is i.i.d.\ with law $\mu$. Indeed, $\Pb(W_k\in A\mid\cF_k)=\mu(A)$ for every $k$ and every Borel set $A$, and iterating this identity factorizes all finite-dimensional distributions.
This constant-kernel case is the setting of the positive theorem. The lower bounds below show that neither symmetry is dispensable within the kernel-oracle class: there are logarithmic counterexamples satisfying Assumption~\ref{ass:additive-kernel-oracle} together with Assumption~\ref{ass:time-homogeneous-noise}, and also logarithmic counterexamples satisfying Assumption~\ref{ass:additive-kernel-oracle} together with Assumption~\ref{ass:state-independent-noise}. Since Assumption~\ref{ass:additive-kernel-oracle} implies Assumption~\ref{ass:centered_additive_l2_noise}, both counterexample classes also satisfy the base second-moment condition.
\paragraph{Random-field interpretation.}
Equivalently, one may imagine that at each round $k$ the oracle samples a fresh random field $(W_{x,k})_{x\in\cX}$ and returns $W_{X_k,k}$. Time-homogeneity corresponds to these fresh fields being sampled from the same law at every time, while state-independence corresponds to all states having the same marginal noise law. The kernel formulation above is used to avoid measurability issues with random fields indexed by a continuum.
\section{Results}

\noindent 
The following theorem gives the positive side of the picture. We state both the pointwise fixed-stepsize estimate and its uniform consequence for standard fixed-stepsize families $\eta=\Theta(1/\sqrt n)$.
\begin{theorem}\label{thm:bound_fh_simpler}
Suppose Assumptions~\ref{ass:feasibility}, \ref{ass:convexity_lipschitzness}, \ref{ass:additive-kernel-oracle}, \ref{ass:time-homogeneous-noise}, and \ref{ass:state-independent-noise} hold, and let $X_1 \equiv x_1$ for some $x_1 \in \cX$. For every $n \geq 2$ and every fixed stepsize $\eta>0$, the iterates generated by \eqref{eq:sgd} satisfy
\begin{align}\label{eq:sgd-bound-eta}
    \bbE[\Delta(X_n)]
&\leq
    \frac{\operatorname{dist}(x_1,\cX_{\star})^2}{2\eta n}
    +
    \left(
    (L + \sigma)^2 + \frac{L^2 + \sigma^2}{2}
    \right)\eta.
\end{align}
In particular, if $\eta = c_n/\sqrt{n}$, with $0<\underline c\le c_n \le \overline c<\infty$, then
\begin{align}\label{eq:sgd-bound-theta}
    \bbE[\Delta(X_n)]
&\leq
    \left[ \frac{\operatorname{dist}(x_1,\cX_{\star})^2}{2 \underline c}
    +
    \left(
    (L + \sigma)^2 + \frac{L^2 + \sigma^2}{2}
    \right)\overline c
    \right] \cdot \frac{1}{\sqrt{n}}.
\end{align}
\end{theorem}
\noindent We make first the following comments and defer the proof to \Cref{sec:upper_bound_proof}.
First, \eqref{eq:sgd-bound-theta} is the precise upper-bound statement for the standard fixed-stepsize regime $\eta=\Theta(1/\sqrt n)$: the hidden constant may depend on $\underline c$ and $\overline c$, but not on $n$. The $1/\sqrt{n}$ dependence in \eqref{eq:sgd-bound-theta} is worst-case optimal up to universal constants. Second, the pointwise estimate \eqref{eq:sgd-bound-eta} is slightly more informative: if one sets $\eta=c/\sqrt n$ and optimizes over $c>0$, then, when $\operatorname{dist}(x_1,\cX_{\star})>0$, the optimizer is
\begin{equation*} \label{eq:opt_fh}
    c_{\star}
=
    \frac{\operatorname{dist}(x_1,\cX_{\star})}
    {\sqrt{3L^2+4L\sigma+3\sigma^2}}
\qquad \implies \qquad
    \bbE[\Delta(X_n)]
\le
    \frac{\operatorname{dist}(x_1,\cX_{\star})\sqrt{3L^2+4L\sigma+3\sigma^2}}{\sqrt n}
\end{equation*}
\noindent To prove Theorem~\ref{thm:bound_fh_simpler}, the first observation is that for every fixed stepsize $\eta>0$ there is at least one index $k_{\star}$ with $1 \le k_{\star} \le n$ such that
\begin{align*}
\bbE[\Delta(X_{k_{\star}})] \le
\frac{\operatorname{dist}(x_1,\cX_{\star})^2}{2\eta n} + \frac{L^2 + \sigma^2}{2}\eta.
\end{align*}
For $\eta=c/\sqrt n$, this becomes
\begin{align*}
\bbE[\Delta(X_{k_{\star}})] \le
\left[ \frac{\operatorname{dist}(x_1,\cX_{\star})^2}{2c} + \frac{L^2 + \sigma^2}{2} \cdot c \right] \cdot \frac{1}{\sqrt{n}}.
\end{align*}
This bound is shown via a technique that is standard in the study of the average iterate of SsGM.
We remark that $k_{\star}$ can be defined as the smallest $k$ in 
$\argmin_{1 \le k \le n} \bbE[\Delta(X_{k})]$, and this makes $k_{\star}$ a deterministic constant once $\cX$, $f$, $x_1$, $\eta$, $n$, the noise law, and the predictable subgradient-selection rule are fixed.
\noindent The first contribution of this paper is showing that, in our setting, 
\begin{align}\label{eq:drift}
\bbE[(\Delta(X_k) - \Delta(X_{k_{\star}}))_{+}]
=
\bbE[(f(X_k) - f(X_{k_{\star}}))_{+}]
&\le (L + \sigma)^2 \eta
& &\text{for $k_{\star} \leq k \le n$,}
\end{align}
which provides a bound of order $\eta$ on how the expected suboptimality can degrade for the iterates following the $k_{\star}$-th. In the standard fixed-stepsize regime $\eta=\Theta(1/\sqrt n)$, this is precisely of order $1/\sqrt n$. In particular, taking $k = n$ proves the desired inequality for the last iterate.

\noindent The proof of \eqref{eq:drift} requires several new ideas. One key step is to argue conditionally on $\cF_{k_{\star}}$, so that we may treat the value of $f(X_{k_{\star}})$ as a known constant, and we may reduce to a simpler problem where we want to deduce \eqref{eq:drift} knowing that for some $q \ge f_{\star}$ we have $f(X_{k_{\star}}) \le q$. This is accomplished thanks to Lemma~\ref{lem:frozen}, where this additional information is leveraged to study $\bbE[(f(X_k) - q)_{+}]$ via the construction of one-sided random sequences that dominate the distance of $X_k$ from a suitable enlargement of the sublevel $\{ f \le q\}$. If a side of the domain is absent or infinite, the corresponding one-sided sequence is either void or lives on $[0,+\infty)$; the latter case is handled by truncation and Fatou's lemma in Corollary~\ref{cor:technical-extended}.
We make use of the theory of general state-space Markov chains, and in particular of a Foster--Lyapunov-type argument~\citep{meyn2009}, to show that the expected excess generated by these excursions remains of order $\eta$, uniformly over the remaining time.

\noindent Under the hypotheses of Theorem~\ref{thm:bound_fh_simpler}, \eqref{eq:sgd-bound-theta} states that every standard fixed-stepsize family $\eta=\Theta(1/\sqrt n)$ has last-iterate error controlled by $C/\sqrt n$, where $C=C(\underline c,\overline c,\operatorname{dist}(x_1,\cX_{\star}),L,\sigma)$ is independent of $n$. Thus this result uses the same interpretation of the $\Theta(1/\sqrt n)$ notation as the lower bound below.
Our second contribution shows that the two kernel symmetries in Theorem~\ref{thm:bound_fh_simpler} are both necessary in a minimax sense, and therefore  Assumption~\ref{ass:centered_additive_l2_noise} alone, is not enough to rule out the logarithm. In particular, the following  shows that the $\log n$ factor is present even if this base condition is strengthened by imposing time-homogeneity alone, or by imposing state-independence alone.
\begin{theorem} \label{thm:clock-counterexamples}
Let $\cX\coloneq [-2,1]$, let $f(x)\coloneq \max\{x,0\}$,
so that
Assumption~\ref{ass:feasibility} holds with $\cX_{\star} \coloneq [-2, 0]$ and Assumption~\ref{ass:convexity_lipschitzness} holds with $L=1$.
Let $g \colon \cX \to [-1, 1]$ be the relative-subgradient selection
associated with $f$ defined by
\begin{equation*}
g(x)\coloneq
\begin{cases}
0, & x\le 0,\\
1, & x > 0.
\end{cases}
\end{equation*}
Fix constants $0<\underline c\le\overline c<\infty$.
Then there exists $n_0=n_0(\underline c,\overline c)\in\bbN$ such that the following holds.
For every $n\ge n_0$ and every $c_n\in[\underline c,\overline c]$, setting $\eta\coloneq c_n/\sqrt n$,
the following hold:
\begin{enumerate}
\item\label{it:state-dependent-lb}
There exists a noise process $(W_k)_{k \in \bbN}$ satisfying Assumptions~\ref{ass:additive-kernel-oracle} and~\ref{ass:time-homogeneous-noise}, but not Assumption~\ref{ass:state-independent-noise}, with $\sigma^2=3$ such that the iterates generated by \eqref{eq:sgd} with some deterministic initialization $X_1\in\cX_\star$ and $G_k\coloneq g(X_k)$ satisfy
\begin{align}
\bbE[f(X_{n})-f_{\star}]
\ge
\frac{1}{512}\,\eta\log\frac1\eta 
\ge
\frac{\underline c}{2^{11}}\frac{\log{n}}{\sqrt{n}}.
\label{eq:sd-main-lower-bound}
\end{align}
\item\label{it:time-inhomogeneous-lb}
There exists a noise process $(W_k)_{k \in \bbN}$ satisfying Assumptions~\ref{ass:additive-kernel-oracle} and~\ref{ass:state-independent-noise}, but not Assumption~\ref{ass:time-homogeneous-noise}, with $\sigma^2 = 1/2$ such that the iterates generated by \eqref{eq:sgd} with $X_1\equiv 0$ and $G_k\coloneq g(X_k)$ satisfy
\begin{align}
\bbE[f(X_{n})-f_{\star}]
\ge
\frac{1}{128}\,\eta\log\frac1\eta
\ge
\frac{\underline c}{2^{11}}\frac{\log{n}}{\sqrt{n}}.
\label{eq:ti-main-lower-bound}
\end{align}
\end{enumerate}
\end{theorem}
\noindent We defer the proof to \Cref{sec:displacement-lower-bounds} and here we make some comments on the lower bound strategy.
The two lower bounds given in Theorem~\ref{thm:clock-counterexamples} implement in different ways a common mechanism to displace the iterates far from the flat optimal set. For labels $j = 1,2,\ldots,M$, with a suitable $M$ of order $1/\eta$ and hence of order $\sqrt n$ for standard fixed stepsizes, the mechanism specifies one rare large noise outcome for each label that is associated with the noise variables. The labels determine for which noise variables this outcome can occur: in the time-inhomogeneous construction, it can be sampled from the noise law at the time when $j$ steps remain, while in the state-dependent construction it can be sampled from the noise law at a special state inside the flat optimal set. The noise outcome associated with label $j$ has magnitude proportional to $j$, and the noise law at the corresponding time or state assigns this outcome probability mass of order $1/j^2$. Thus, when this outcome is sampled, the SsGM update displaces the iterate by order $\eta j$. In the state-dependent construction, the non-rare noise outcomes move the iterate from one labeled state to the next, and, along the path where the earlier noise samples take these non-rare values, the state labeled by $j$ is reached when about $j$ steps remain. In both constructions, sampling the rare large outcome associated with label $j$ produces terminal error of order $\eta j$, so its expected contribution is of order $\eta/j$. Summing over $j = 1,2,\ldots,M$ gives $\eta\log(1/\eta)$, which is of order $(\log n)/\sqrt n$ uniformly over all $c_n\in[\underline c,\overline c]$.

\section{Proof of Theorem~\ref{thm:bound_fh_simpler}}\label{sec:upper_bound_proof}
\noindent The proof is based on a monotone quantitative estimate on how the contribution of the term $- \eta \cdot G_k$ in \eqref{eq:sgd} can move $X_k$ toward or beyond a certain threshold separating $X_k$ from $\cX_{\star}$. This will involve the object introduced in the following definition:
\begin{definition}\label{def:T}
Let $D\in \bbR_{\ge 0} \cup \{+\infty\}$, let $\cI \coloneq [0,D]$, let $m: \cI \to \bbR_{\ge 0}$ be non-decreasing and let $\eta >0$. We define
\begin{equation*}
T_{\cI, m, \eta}(z)\coloneq 
\sup_{0\le s\le z} \lrb{s - \eta \cdot m(s)}_{+},
\qquad z\in\cI,
\end{equation*}
where $(u)_{+}\coloneq \max\{u,0\}$.
\end{definition}
\noindent We study some crucial properties of the object defined in Definition~\ref{def:T}
\begin{proposition}
In the setting of Definition~\ref{def:T}:
\begin{enumerate}
\item \label{it:T-non-decr} $T_{\cI, m, \eta}$ is non-decreasing;
\item \label{it:T(z)-le-z} The inequalities $0 \le T_{\cI, m, \eta}(z) \le z$ hold for every $z \in \cI$;
\item \label{it:T-lip} $T_{\cI, m, \eta}$ is $1$-Lipschitz on $\cI$;
\item \label{it:cornerstone} For every $z \in \cI$,
\begin{equation}\label{eq:cornerstone-ineq}
 \eta \cdot \int_{0}^z m(s)\dif s \le \bigl(z-T_{\cI,m,\eta}(z)\bigr)\lrb{z+\eta \cdot m(z)} .
\end{equation}
\end{enumerate}
\end{proposition}
\begin{proof}
\noindent For brevity, set $T \coloneq T_{\cI, m, \eta}$. 

\noindent \ref{it:T-non-decr}. Follows directly by the definition.

\noindent \ref{it:T(z)-le-z}. For every $0\le s\le z$, we have $0\le (s-\eta m(s))_{+}\le s\le z$. Taking the supremum over $0\le s\le z$ gives $0\le T(z)\le z$.

\noindent \ref{it:T-lip}. Since $T$ is non-decreasing, it is enough to prove that for every $x$, $y \in \cI$, with $x \le y$,
\begin{equation}\label{eq:T-is-lip}
T(y)-T(x)\le y-x.
\end{equation}
If $x=y$, this is trivial. If $x < y$, note that
\begin{equation}\label{eq:T-identity}
T(y)=\max\left\{T(x),\sup_{x<s\le y}\lrb{s-\eta m(s)}_{+}\right\}.
\end{equation}
For every $x<s\le y$, monotonicity of $m$ and the definition of $T(x)$ yield
\begin{equation*}
s-\eta m(s)
\le y-\eta m(x)
\le y-x+x-\eta m(x)
\le y-x+T(x).
\end{equation*}
Thus $\lrb{s-\eta m(s)}_{+}\le y-x+T(x)$, and \eqref{eq:T-identity} gives \eqref{eq:T-is-lip}. 

\noindent \ref{it:cornerstone}. If $T(z) = 0$, then \eqref{eq:cornerstone-ineq} follows by $\int_{0}^z m(s) \dif s \le z \cdot m(z)$. If $T(z) > 0$, then for every $0 \le s^{*} \le z$, monotonicity of $m$ gives
\begin{equation*}
\int_{0}^z m(s) \dif s = \int_{0}^{s^*} m(s) \dif s + \int_{s^{*}}^z m(s) \dif s \le m(s^{*}) \cdot s^{*} + m(z) \cdot (z - s^{*}) .
\end{equation*}
Multiplying by $\eta$ and using that $m(s^{*}) \ge 0$ and $z \ge s^{*}$, we get
\begin{align*} 
\eta \cdot \int_{0}^z m(s) \dif s
&\le
 \eta \cdot m(s^{*}) \cdot s^{*} + \eta \cdot m(z) \cdot (z - s^{*})\\
&=
 (s^{*} - s^{*} + \eta \cdot m(s^{*})) \cdot s^{*}+(z - s^{*} + \eta \cdot m(s^{*}) - \eta \cdot m(s^{*})) \cdot \eta \cdot m(z)\\
&\le
(z - s^{*} + \eta \cdot m(s^{*})) \cdot z + (z - s^{*} + \eta \cdot m(s^{*})) \cdot \eta \cdot m(z)\\
&=
\Big( z - \big(s^{*} - \eta \cdot m(s^{*})\big)\Big) \cdot \lrb{z + \eta \cdot m(z)}.
\end{align*}
Then,
\begin{align*}
\eta \cdot \int_{0}^z m(s) \dif s
&\le \inf_{0 \le s^{*} \le z} \Big( z - \big(s^{*} - \eta \cdot m(s^{*})\big)\Big)\cdot \lrb{z + \eta \cdot m(z)}\\
&= \lrb{ z - \sup_{0 \le s^{*} \le z}(s^{*} - \eta \cdot m(s^{*})) }\cdot \lrb{z + \eta \cdot m(z)}\\
&= (z - T(z)) \cdot \lrb{z + \eta \cdot m(z)},
\end{align*}
where the last step uses $T(z) > 0$.
\end{proof}

\noindent The next proposition provides the essential technical result to prove the key subsequent Lemma~\ref{lem:frozen}.
\begin{proposition}\label{prop:technical}
Let $\cI = [0, D]$ be a closed bounded interval, let $m \colon \cI \to [0, L]$ be non-decreasing, and let $\eta > 0$.
Let $R$ ,$R_1$, $R_2$, $\dots$ be an i.i.d.\ sequence of random variables such that
$\bbE[R]= 0$ and $\bbE[R^2] \le \sigma^2$.
Define the sequence $(Z_k)_{k \in \bbN}$ by
\begin{align}\label{eq:chain}
\begin{cases}
Z_1 \equiv 0,\\
Z_{k + 1} = \Pi_{\cI} (T_{\cI, m, \eta}(Z_k) - \eta \cdot R_k ) & k \geq 1.
\end{cases} 
\end{align}
Then:
\begin{equation*}
\forall k \ge 1, \qquad 
\bbE\lsb{\int_0^{Z_k} m(s) \dif s} \le \lrb{\frac{\sigma^2}{2} + L \sigma}\cdot \eta .
\end{equation*}
\end{proposition}
\begin{proof}
\noindent Once again, set $T \coloneq T_{\cI, m, \eta}$. We now discuss \eqref{eq:chain} in the framework of general state space Markov chains. The reader may find in Section~\ref{sec:markov} some basic notions and standard notation related to this topic and transition probability kernels.
Consider the map $\Phi \colon \cI \times \bbR \to \cI$ given by
\begin{equation*}\label{eq:map-increment-noise}
\Phi(z, r) \coloneq \Pi_{\cI} (T(z) - \eta \cdot r), 
\end{equation*}
which is continuous by continuity of $T$ and of the projection $\Pi_\cI$.
Let $\mu \coloneq \law(R)$ be the law of $R$, write $\borel(\cI)$ for the Borel $\sigma$-field of $\cI$, and note that the map $P \colon \cI \times \borel(\cI) \to [0, 1]$ defined by
\begin{equation*}\label{eq:kernel}
P(z, \cA) \coloneq \int_{\bbR} \bbI \lcb{\Phi(z, r) \in \cA } \dif\mu(r) = \bbE[ \bbI\{\Phi(z, R) \in \cA\}]
\end{equation*}
is a transition probability kernel. Moreover, $Z_{k+1}=\Phi(Z_k,R_k)$, and the i.i.d.\ hypothesis on $(R_k)_{k \in \bbN}$ shows that $(Z_k)_{k \in \bbN}$ is a time-homogeneous Markov chain on the compact state space $\cI$ with transition probability kernel $P$ and initial distribution $\delta_0$.
\noindent We assume the two following claims:
\begin{claim}\label{claim:invariant}
There exists a probability measure $\pi$ over $(\cI, \borel(\cI))$ that is invariant for the chain, i.e.,
\begin{equation}\label{eq:inv-measure}
 \pi P = \pi.
\end{equation}
\end{claim}
\begin{claim}\label{claim:BC}
There exists a bounded non-negative measurable function $B \colon \cI \to \bbR$ and $C \in \bbR$ that satisfy the inequality
\begin{equation}\label{eq:fl-type-ineq}
P B(z) - B(z) \le -\int_{0}^{z} m(s) \dif s + C \cdot \eta,
\qquad
z \in \cI.
\end{equation}
More specifically, a solution is given by
\begin{equation}\label{eq:BC}
B(z)
\coloneq
\frac{z^2}{2 \eta} + 2 L z
\qquad
C \coloneq \frac{\sigma^2}{2} + L \sigma.
\end{equation}
\end{claim}
\noindent Let $B$ and $C$ be as in \eqref{eq:BC}, and let $\pi$ be the invariant measure satisfying \eqref{eq:inv-measure}. Then,
\begin{equation*}
\int_{\cI} \pi(\dif z) P B(z) = \int_{\cI} \pi(\dif z) B(z).
\end{equation*}
Therefore, integrating \eqref{eq:fl-type-ineq} with respect to $\pi$ gives
\begin{equation}\label{eq:int-the-int}
\int_{\cI} \pi(\dif z) 
\lrb{\int_{0}^{z} m(s) \dif s}
\le
\lrb{\frac{\sigma^2}{2} + L \sigma}
\cdot
\eta.
\end{equation}
\noindent To conclude the proof, define $(Y_k)_{k \in \bbN}$ by taking $Y_1 \sim \pi$ independent of $(R_{k})_{k \in \bbN}$ and setting
\begin{equation*}
\forall k \ge 1,\qquad Y_{k + 1} = \Phi(Y_k, R_k).
\end{equation*}
Then $Y_k\sim\pi$ for every $k$ by definition of $\pi$.
Note that $Z_1 = 0 \le Y_1 \in \cI$, and we can show by induction that $Z_k \le Y_k$ for every $k \ge 1$, since for every fixed $r$ the map $s \mapsto \Phi(s,r)$ is non-decreasing, hence
\begin{equation*}
\forall k \ge 1, \qquad Z_k \le Y_k \quad \implies \quad Z_{k+1} =\Phi(Z_k, R_k) \le \Phi(Y_k, R_k) = Y_{k+1}. 
\end{equation*}
Using that $z\mapsto\int_0^z m(s)\dif s$ is also non-decreasing, it follows that
\begin{align*}
\bbE\lsb{\int_0^{Z_k} m(s) \dif s}
&\le
\bbE\lsb{\int_0^{Y_k} m(s) \dif s}
=\int_{\cI} \pi(\dif z) \lrb{\int_{0}^{z} m(s) \dif s}
\le
\lrb{\frac{\sigma^2}{2} + L \sigma}\eta,
\end{align*}
where the last inequality is \eqref{eq:int-the-int}.
\end{proof}
\noindent We can extend Proposition~\ref{prop:technical} to the case $D = + \infty$.

\begin{corollary}\label{cor:technical-extended}
Let $\cI\coloneq[0, +\infty]$, let $m\colon\cI\to[0,L]$ be non-decreasing, and let $\eta > 0$.
Let $R$, $R_1$, $R_2$,~$\dots$ be an i.i.d.\ sequence of random variables 
such that $\bbE[R]= 0$ and $\bbE[R^2] \le \sigma^2$. Define $(Z_k)_{k \in \bbN}$ by
\begin{align*}
\begin{cases}
Z_1\equiv0,\\
Z_{k+1}=\Pi_\cI\lrb{T_{\cI,m,\eta}(Z_k) - \eta R_k} & k \ge 1,
\end{cases}
\end{align*}
Then
\begin{equation*}
\forall k \ge 1, \qquad \bbE\lsb{\int_0^{Z_k}m(s)\dif s}
\le
\lrb{\frac{\sigma^2}{2}+L\sigma}\eta .
\end{equation*}
\end{corollary}
\begin{proof}
Fix an integer $N \ge 1$. For $M >0$, let $\cI_M=[0,M]$, let $T_M\coloneq T_{\cI_M,m_{|_{\cI_M}},\eta}$, and define the truncated chain
\begin{align*}
\begin{cases}
Z^{(M)}_1\equiv0,\\
Z^{(M)}_{k+1}=\Pi_{\cI_M}\lrb{T_M(Z^{(M)}_k) - \eta R_k} & k \ge 1.
\end{cases}
\end{align*}
By Proposition~\ref{prop:technical},
\begin{equation}\label{eq:truncated-bound}
\bbE\lsb{\int_0^{Z^{(M)}_N}m(s)\dif s}
\le
\lrb{\frac{\sigma^2}{2}+L\sigma}\eta
\qquad\text{for all $M>0$.}
\end{equation}
On the event $\{ \max_{1 \le k \le N} Z_k < M\}$, the truncated and untruncated chains coincide up to time $N$. Indeed, by induction, whenever $Z^{(M)}_k=Z_k<M$, we have $T_M(Z_k)=T_{\cI,m,\eta}(Z_k)$; moreover, since the untruncated next state is still below $M$, the projections onto $[0,M]$ and $[0,+\infty)$ agree at the next step.
Hence $Z^{(M)}_N\to Z_N$ almost surely as $M\to+\infty$.
Since $z\mapsto\int_0^z m(s)\dif s$ is non-negative, Fatou's lemma and \eqref{eq:truncated-bound} yield
\begin{equation*}
\bbE\lsb{\int_0^{Z_N}m(s)\dif s}
\le
\liminf_{M\to+\infty}
\bbE\lsb{\int_0^{Z^{(M)}_N}m(s)\dif s}
\le
\lrb{\frac{\sigma^2}{2}+L\sigma}\eta.
\end{equation*}
Since $N$ was arbitrary, the proof is complete.
\end{proof}

\noindent We can now prove a key lemma used to prove our positive result.
\begin{lemma}\label{lem:frozen}
Suppose Assumptions~\ref{ass:feasibility}, \ref{ass:convexity_lipschitzness}, \ref{ass:additive-kernel-oracle}, \ref{ass:time-homogeneous-noise}, and~\ref{ass:state-independent-noise} hold, and let $(X_k)_{k \ge 1}$ be generated by \eqref{eq:sgd}.
Assume the existence of a positive integer $k_{\star}$ and some $q\in\bbR$ such that $f(X_{k_{\star}})\le q$, deterministically.
Then, for every $k \ge k_{\star}$,
\begin{equation*}\label{eq:bounding-excess}
\bbE[(f(X_k) - q)_{+}] \le (L + \sigma)^2\cdot \eta.
\end{equation*}
\end{lemma}
\begin{proof}
Let
\begin{equation*}
S_q\coloneq\{x\in\cX:f(x)\le q\}.
\end{equation*}
The set $S_q$ is non-empty because $f(X_{k_{\star}})\le q$, and it is a closed convex subset of $\cX$. Set
\begin{equation*}
a\coloneq\inf\cX,
\qquad
b\coloneq\sup\cX,
\end{equation*}
where the endpoints are allowed to be infinite, and (possibly) enlarge $S_q$ to
\begin{equation*}\label{eq:enlarged}
E_q
\coloneq
\lcb{x\in\cX:\operatorname{dist}(x,S_q)\le\eta L}.
\end{equation*}
This is a non-empty closed interval in $\cX$, possibly unbounded on either side. Let
\begin{equation*}
\alpha\coloneq\inf E_q,
\qquad
\beta\coloneq\sup E_q,
\end{equation*}
so that $E_q=[\alpha,\beta]\cap\cX$, using our extended-endpoint convention for intervals. We say that the left branch is present if $\alpha>a$, and that the right branch is present if $\beta<b$. If a branch is not present, its contribution below is understood to be identically zero.

\noindent Suppose first that the left branch is present. Then $\alpha$ is finite, and we define 
\[
    D^{(1)}\coloneq\alpha-a,
\]
we note that $0 < D^{(1)} \leq +\infty$, with $D^{(1)}=+\infty$ when $a=-\infty$, and we define
\begin{equation*}
\cI^{(1)}\coloneq[0,D^{(1)}].
\end{equation*}
Let
\begin{equation*}
z^{(1)}:\cX\to\cI^{(1)},
\qquad
z^{(1)}(x)\coloneq(\alpha-x)_{+},
\end{equation*}
and
\begin{equation*}\label{eq:h-q-alpha}
h^{(1)}:\cI^{(1)}\to\bbR_{\ge 0},
\qquad
h^{(1)}(z)\coloneq f(\alpha-z)-f(\alpha).
\end{equation*}
The function $h^{(1)}$ is convex and $L$-Lipschitz by composition with the affine map $z\mapsto\alpha-z$. It is non-negative and non-decreasing because, on the left of the sublevel set $S_q$, convexity implies that all relative subgradients of $f$ are non-positive; after the change of variables $x=\alpha-z$, these slopes become non-negative slopes for $h^{(1)}$. Define $m^{(1)}(0)=0$ and, for $z>0$, let $m^{(1)}(z)$ be the left derivative of $h^{(1)}$ at $z$. Then $m^{(1)}:\cI^{(1)}\to[0,L]$ is non-decreasing and
\begin{equation*}\label{eq:m-q-alpha-left}
h^{(1)}(z)=\int_0^z m^{(1)}(s)\dif s,
\qquad z\in\cI^{(1)}.
\end{equation*}
Similarly, if the right branch is present, then $\beta$ is finite, and we define 
\[
    D^{(2)}\coloneq b - \beta,
\]
where $0 < D^{(2)} \leq +\infty$, with $D^{(2)}=+\infty$ when $b=+\infty$, and we define
\begin{equation*}
\cI^{(2)}\coloneq[0,D^{(2)}],
\end{equation*}
together with
\begin{equation*}
z^{(2)}:\cX\to\cI^{(2)},
\qquad
z^{(2)}(x)\coloneq(x-\beta)_{+},
\end{equation*}
and
\begin{equation*}\label{eq:h-q-beta}
h^{(2)}:\cI^{(2)}\to\bbR_{\ge 0},
\qquad
h^{(2)}(z)\coloneq f(\beta+z)-f(\beta).
\end{equation*}
Again, $h^{(2)}$ is convex and $L$-Lipschitz. It is non-negative and non-decreasing because, on the right of $S_q$, all relative subgradients of $f$ are non-negative. Define $m^{(2)}(0)=0$ and, for $z>0$, let $m^{(2)}(z)$ be the left derivative of $h^{(2)}$ at $z$. Then $m^{(2)}:\cI^{(2)}\to[0,L]$ is non-decreasing and
\begin{equation*}\label{eq:m-q-alpha-right}
h^{(2)}(z)=\int_0^z m^{(2)}(s)\dif s,
\qquad z\in\cI^{(2)}.
\end{equation*}

\noindent For each present branch $i\in\{1,2\}$, write $T^{(i)}=T_{\cI^{(i)},m^{(i)},\eta}$ and define $(Z_k^{(i)})_{k \in \bbN}$ by
\begin{align}\label{eq:branch-chain}
\begin{cases}
Z_k^{(i)}=0 & 1\le k\le k_{\star},\\
Z_{k+1}^{(i)}=\Pi_{\cI^{(i)}}\lrb{T^{(i)}(Z_k^{(i)})-(-1)^i\eta W_k} & k\ge k_{\star}.
\end{cases}
\end{align}

\begin{claim}\label{claim:sandwich} 
For every present branch $i\in\{1,2\}$, a.s.,
\begin{equation*}\label{eq:stochastic-bound}
z^{(i)}(X_k)\le Z_k^{(i)}
\qquad\text{for all $k\ge k_{\star}$.}
\end{equation*}
\end{claim}
\noindent Claim~\ref{claim:sandwich} is proved in Section~\ref{sec:claim-proofs}. We now prove the deterministic decomposition
\begin{equation}\label{eq:bound-a-b-pointwise}
(f(x)-q)_{+}
\le
\eta L^2
+
\sum_{i\in\{1,2\}:\,\text{branch }i\text{ present}}
h^{(i)}\lrb{z^{(i)}(x)},
\qquad x\in\cX.
\end{equation}
If $x\in E_q$, then $\operatorname{dist}(x,S_q)\le\eta L$, hence by Lipschitzness $(f(x)-q)_{+}\le\eta L^2$. If $x<\alpha$, then the left branch is present. Since $\alpha\in E_q$, there exists $y\in S_q$ such that $|y-\alpha|\le\eta L$, and therefore
\begin{equation*}
(f(x)-q)_{+}
\le f(x)-f(y)
\le f(x)-f(\alpha)+L|y-\alpha|
\le h^{(1)}(z^{(1)}(x)) + \eta L^2.
\end{equation*}
The case $x>\beta$ is symmetric and gives
\begin{equation*}
(f(x)-q)_{+}
\le h^{(2)}(z^{(2)}(x)) + \eta L^2.
\end{equation*}
Since every $x\in\cX$ belongs to exactly one of these three regions, \eqref{eq:bound-a-b-pointwise} follows.

\noindent Let $k\ge k_{\star}$. By \eqref{eq:bound-a-b-pointwise}, Claim~\ref{claim:sandwich}, and monotonicity of the functions $h^{(i)}$,
\begin{equation*}
\bbE[(f(X_k)-q)_{+}]
\le
\eta L^2
+
\sum_{i\in\{1,2\}:\,\text{branch }i\text{ present}}
\bbE\lsb{h^{(i)}(Z_k^{(i)})}.
\end{equation*}
For a present branch $i$, set $\widetilde Z_h\coloneq Z_{k_\star+h-1}^{(i)}$ and $R_h\coloneq(-1)^iW_{k_{\star}+h-1}$.
Then $(R_h)_{h \ge 1}$ is i.i.d., centered, and satisfies $\bbE[R_h^2]\le\sigma^2$. If $D^{(i)}<+\infty$, Proposition~\ref{prop:technical} applies to $(\widetilde Z_h)_{h\ge1}$. If $D^{(i)}=+\infty$, Corollary~\ref{cor:technical-extended} applies instead. Therefore, in either case,
\begin{equation*}
\bbE\lsb{h^{(i)}(Z_k^{(i)})}
=
\bbE\lsb{\int_0^{Z_k^{(i)}}m^{(i)}(s)\dif s}
\le
\lrb{\frac{\sigma^2}{2}+L\sigma}\eta.
\end{equation*}
There are at most two present branches, and therefore
\begin{equation*}
\bbE[(f(X_k)-q)_{+}]
\le
\eta L^2+2\lrb{\frac{\sigma^2}{2}+L\sigma}\eta
=(L+\sigma)^2\eta. \qedhere
\end{equation*}
\end{proof}
\noindent Now we have all the ingredients to complete the proof of Theorem~\ref{thm:bound_fh_simpler}.
Let $x_{\star}\in\cX_{\star}$ be such that $|x_1-x_{\star}|=\operatorname{dist}(x_1,\cX_{\star})$. For every $k \ge 1$, by non-expansiveness of the projection onto $\cX$,
\begin{align*}
\lrb{X_{k+1} - x_{\star}}^2 &= \lrb{\Pi_{\cX} (X_k - \eta\cdot (W_k + G_k)) - \Pi_{\cX}(x_{\star})}^2
\le \lrb{X_k - \eta\cdot (W_k + G_k) - x_{\star}}^2 .
\end{align*} 
By Assumptions~\ref{ass:additive-kernel-oracle}, \ref{ass:time-homogeneous-noise}, and~\ref{ass:state-independent-noise}, there exists a centered law $\mu$ with second moment at most $\sigma^2$ such that $\Pb(W_k\in A\mid\cF_k)=\mu(A)$ for every $k$ and every Borel set $A$. Since $G_k$ and $X_k$ are $\cF_k$-measurable, it follows that
\begin{align*}
\bbE[\lrb{X_{k+1} - x_{\star}}^2]
&\le
\bbE[\lrb{X_k - \eta\cdot (W_k + G_k) - x_{\star}}^2]\\
&\le\bbE[\lrb{X_k - \eta\cdot G_k - x_{\star}}^2]+\sigma^2\eta^2\\
&= \bbE[\lrb{X_k - x_{\star}}^2] - 2\bbE[G_k(X_k - x_{\star})] \cdot \eta + \lrb{ \bbE[G_k^2] + \sigma^2 }\cdot \eta^2\\
&\le \bbE[\lrb{X_k - x_{\star}}^2] - 2\bbE[G_k(X_k - x_{\star})] \cdot \eta + ( L^2 + \sigma^2 )\cdot \eta^2.
\end{align*}
The convexity of $f$ entails that $G_k(X_k - x_{\star}) \ge f(X_k) - f(x_{\star}) = \Delta(X_k)$. Hence
\begin{equation}\label{eq:telescopic}
2 \eta \cdot \bbE[\Delta(X_k)] \le \bbE[\lrb{X_k - x_{\star}}^2] - \bbE[\lrb{X_{k + 1} - x_{\star}}^2] + (L^2 + \sigma^2 )\cdot \eta^2.
\end{equation}
We sum \eqref{eq:telescopic} for $k=1,2,\dots,n$ to get%
\footnote{Note that here we assume the existence of $X_{n + 1}$, despite $X_n$ being called the \emph{last} iterate. The underlying tacit assumption, that was evident also in \eqref{eq:sgd}, is that the last iterate is called last in the sense that it is the last that we want to compute in implementations, but we assume that the recurrence defines the iterates for every $k \ge 1$.}
\begin{align*}
2 \eta \cdot \sum_{k = 1}^{n} \bbE[\Delta(X_k)]
&\le |x_1-x_{\star}|^2 - \bbE[\lrb{X_{n+1} - x_{\star}}^2] + n \cdot (L^2 + \sigma^2 ) \cdot \eta^2\\
&\le \operatorname{dist}(x_1,\cX_{\star})^2 + n \cdot (L^2 + \sigma^2 )\cdot \eta^2 .
\end{align*}
Therefore, there exists an index $k_{\star}\in\{1,\dots,n\}$ such that
\begin{equation}\label{eq:bound_k_star}
\bbE[\Delta(X_{k_{\star}})]
\le 
\frac{\operatorname{dist}(x_1,\cX_{\star})^2}{2 \eta n}
+
\frac{L^2 + \sigma^2}{2}\eta.
\end{equation}
By \eqref{eq:bound_k_star}, for every $k_{\star}\le k\le n$,
\begin{align*}
\bbE[\Delta(X_k)]
&= \bbE[\Delta(X_{k_{\star}})] + \bbE[\Delta(X_k) - \Delta(X_{k_{\star}})]\\
&\le \bbE[\Delta(X_{k_{\star}})] + \bbE[(\Delta(X_k) - \Delta(X_{k_{\star}}))_{+}]\\
&\le \frac{\operatorname{dist}(x_1,\cX_{\star})^2}{2 \eta n} + \frac{L^2 + \sigma^2 }{2}\eta + \bbE[(\Delta(X_k) - \Delta(X_{k_{\star}}))_{+}]
\\
&= \frac{\operatorname{dist}(x_1,\cX_{\star})^2}{2 \eta n} + \frac{L^2 + \sigma^2 }{2}\eta + \bbE[(f(X_k) - f(X_{k_{\star}}))_{+}] \eqcolon (\star)
\end{align*}
What is left is bounding $\bbE[(f(X_k) - f(X_{k_{\star}}))_{+}]$. To this end, we make the following claim.
\begin{claim}\label{claim:freezing}    
For every $k_{\star}\le k\le n$,
\begin{equation}\label{eq:post-good-iterate-deterioration}
\bbE\left[
(f(X_k)-f(X_{k_{\star}}))_{+}
\mid \cF_{k_{\star}}
\right]
\le
(L+\sigma)^2\eta
\qquad\text{a.s.}
\end{equation}
\end{claim}
The insight into Claim~\ref{claim:freezing} is that $X_{k_{\star}}$ is $\cF_{{k_{\star}}}$-measurable and that there is independence between the  $\sigma$-fields $\cF_{{k_{\star}}}$ and $\sigma(W_{k_{\star}}, W_{k_{\star} + 1}, \dots, W_{k - 1})$, hence $X_{k_{\star}}$ can be treated as a constant when taking the conditional expectation of $(f(X_k) - f(X_{k_{\star}}))_{+}$, with respect to $\cF_{{k_{\star}}}$. This justifies a conditional application of  Lemma~\ref{lem:frozen} with $q = f(X_{k_{\star}})$, which gives \eqref{eq:post-good-iterate-deterioration}. The interested reader can find a complete proof of this claim in Section~\ref{sec:claim-proofs}.

\noindent An application of Claim~\ref{claim:freezing} and the tower property of expectation yields
\begin{align*}
(\star)
&= \frac{\operatorname{dist}(x_1,\cX_{\star})^2}{2 \eta n} + \frac{L^2 + \sigma^2 }{2}\eta + \bbE[\bbE[(f(X_k) - f(X_{k_{\star}}))_{+} | \cF_{k_{\star}}]]\\
&\le \frac{\operatorname{dist}(x_1,\cX_{\star})^2}{2 \eta n} + \frac{L^2 + \sigma^2 }{2}\eta + (L + \sigma)^2 \eta.
\end{align*}
Taking $k=n$ gives \eqref{eq:sgd-bound-eta}. Substituting $\eta=c_n/\sqrt n$, with $c_n\in[\underline c,\overline c]$, then \eqref{eq:sgd-bound-theta} follows immediately from \eqref{eq:sgd-bound-eta}. This completes the proof.

\section{Proof of Theorem~\ref{thm:clock-counterexamples}}\label{sec:displacement-lower-bounds}
Fix constants $0<\underline c\le\overline c<\infty$, a horizon $n$, and a coefficient $c_n\in[\underline c,\overline c]$. Set
\begin{align}
\eta &\coloneq \frac{c_n}{\sqrt n},
&
M&\coloneq\left\lfloor\frac{1}{16\eta}\right\rfloor,
\label{eq:clock-common-parameters}
&
\lambda&\coloneq\min\left\{\eta,\frac{1}{n\eta}\right\}.\notag
\end{align}
In the proof we take $n_0=n_0(\underline c,\overline c)$ large enough so that, for all $n\ge n_0$ and all $c_n\in[\underline c,\overline c]$,
\begin{equation}\label{eq:theta-lb-size-conditions}
\eta\le 2^{-8},
\qquad
\eta\ge\frac1n,
\qquad
\log\frac{\sqrt n}{\overline c}\ge\frac14\log n.
\end{equation}
For example, it is enough to take $n_0$ larger than a numerical constant times $\max\{\overline c^2,\underline c^{-2},\overline c^4,1\}$.
Then $M\ge16$ and $M\le n-1$. Moreover,
\begin{equation}\label{eq:lambda-delta-basic}
\lambda\le\eta,
\qquad
\lambda^2\le\frac1n,
\qquad
\eta \lambda\le\eta,
\qquad
n\eta \lambda\le1.
\end{equation}
Indeed, $\lambda\le\eta$ by definition, while $\lambda\le1/\sqrt n$ follows from $\lambda=\min\{\eta,1/(n\eta)\}$.
Also, if $\lambda=\eta$, then $n\eta \lambda=n\eta^2\le1$, whereas if $\lambda=1/(n\eta)$, then $n\eta \lambda=1$.

Recall from the statement of Theorem~\ref{thm:clock-counterexamples} that:
\begin{equation}\label{eq:clock-common-objective}
\cX=[-2,1],
\qquad
f(x)=x_{+}=\max\{x,0\},
\qquad
f_\star=0,
\qquad
\cX_\star=[-2, 0],
\end{equation}
and
\begin{equation}\label{eq:clock-common-g-selection}
g(x)=
\begin{cases}
0, & x \le 0,\\
1, & x>0.
\end{cases}
\end{equation}
Then $f$ is convex and $1$-Lipschitz on $\cX$, and $g(x)\in\partial_{\cX}f(x)\cap[-1,1]$ for every $x\in\cX$.

The next lemma reveals the harmonic structure emerging from our lower bound instance which leads to the logarithm term appearing in the lower bound.
\begin{lemma}\label{lem:clock-template}
Let $Y$ be a non-negative random variable. Let $\eta\in(0,2^{-8}]$ and let $M\coloneq\lfloor1/(16\eta)\rfloor$.
Suppose that there exist pairwise disjoint events $E_1,\ldots,E_M$ and constants $p,\ell>0$ such that, for every $j = 1,\ldots,M$,
\begin{equation*}
\Pb(E_j)\ge \frac{p}{j^2},
\qquad
Y\ge \ell\eta j\quad\text{on }E_j.
\end{equation*}
Then
\begin{equation*}
\bbE[Y]\ge \frac{p\ell}{2}\,\eta\log\frac1\eta.
\end{equation*}
\end{lemma}

\begin{proof}
By disjointness and non-negativity,
\begin{equation*}
\bbE[Y]
\ge
\sum_{j = 1}^{M}\ell\eta j\,\Pb(E_j)
\ge
p\ell\eta \sum_{j = 1}^{M}\frac{1}{j}.
\end{equation*}
Since $\sum_{j = 1}^{M}j^{-1} \ge \log(M+1)$ and $M+1>1/(16\eta)$,
\begin{equation*}
\sum_{j = 1}^{M}\frac{1}{j}
\ge
\log\frac1{16\eta}
=
\log\frac1\eta-\log16.
\end{equation*}
For $\eta\le2^{-8}$, the last display is at least $\frac12\log(1/\eta)$, and the claim follows.
\end{proof}

The following lemma explains why a rare large outcome of magnitude proportional to $j$, sampled when $j$ steps remain, leaves terminal error of order $\eta j$.
\begin{lemma}\label{lem:clock-payoff}
Fix $j \in \{1,\ldots,M\}$ and put $k_j\coloneq n-j$. Suppose that, on some realization,
\begin{equation*}
-(j+1)\eta \lambda \le X_{k_j} \le 0,
\qquad
W_{k_j}=-4j,
\qquad
W_k=0\quad\text{for every }k = k_j+1,\ldots,n-1.
\end{equation*}
Then, on the same realization,
\begin{equation*}
f(X_n)-f_\star \ge 2\eta j.
\end{equation*}
\end{lemma}
\begin{proof}
Since $X_{k_j} \le 0$, we have $g(X_{k_j})=0$. Hence $W_{k_j}=-4j$ gives
\begin{equation*}
X_{k_j+1}
=
\Pi_\cX\bigl(X_{k_j}-\eta W_{k_j}\bigr)
=
\Pi_\cX\bigl(X_{k_j}+4\eta j\bigr).
\end{equation*}
Because $X_{k_j}\in[-(j+1)\eta \lambda, 0]$, $\eta \lambda\le\eta$, and $j\le M\le1/(16\eta)$,
\begin{equation*}
0< 4\eta j-(j+1)\eta \lambda
\le
X_{k_j}+4\eta j
\le
4\eta j
\le
\frac{1}{4}.
\end{equation*}
Thus the projection does not act and
\begin{equation*}
X_{k_j+1} \ge 4\eta j-(j+1)\eta \lambda>0.
\end{equation*}
By hypothesis, for the remaining $j-1$ updates the noises vanish.
At each such update, the deterministic term can decrease the iterate by at most $\eta$, because $g(x)\in[0,1]$.
Since $X_{k_j+1}\le 1/4$ and the subsequent noiseless update $x\mapsto x-\eta g(x)$ never increases the pre-projection point, all subsequent pre-projection points are at most $1$.
Hence the projection onto $[-2,1]$ cannot decrease them.
Therefore
\begin{align*}
X_n
&\ge
4\eta j-(j+1)\eta \lambda-(j-1)\eta
\ge
4\eta j-(j+1)\eta-(j-1)\eta
=
2\eta j.
\end{align*}
In particular $X_n>0$, and so $f(X_n)=X_n$.
\end{proof}

\subsection{A time-homogeneous state-dependent implementation}\label{subsec:state-dependent-clock}
This subsection proves item~\ref{it:state-dependent-lb} of Theorem~\ref{thm:clock-counterexamples}.

\noindent For $j = 0,1,\dots,n-1$, define the labeled states
\begin{equation*}
s_j \coloneq -(j+1)\eta \lambda.
\end{equation*}
Set $X_1\equiv s_{n-1}=-n\eta \lambda$. By \eqref{eq:lambda-delta-basic}, $X_1\in[-1,0]\subseteq\cX_\star$. Thus $s_0=-\eta \lambda$.
We define the noise law at state $s_j$ depending on whether $s_j$ is near or far from $0$.

\paragraph{Far labeled states.}
For $j = M+1,\dots,n-1$, define
\begin{equation*}\label{eq:sd-far-kernel}
Q(s_j,\cdot)
=
(1-\lambda^2)\delta_{-\lambda}
+
\lambda^2\delta_{\frac{1-\lambda^2}{\lambda}}.
\end{equation*}
The noise realization $-\lambda$ moves the point from $s_j$ to $s_{j-1}$, because
\begin{equation*}
    s_j-\eta(-\lambda)
=
    -(j+1)\eta \lambda+\eta\lambda
=
    -j\eta \lambda
=
    s_{j-1}.
\end{equation*}

\paragraph{Near labeled states.}
For $j = 1,\dots,M$, set
\begin{equation*}\label{eq:sd-near-parameters}
a_j = \frac{1}{64j^2},
\qquad
b_j = \lambda^2+\frac{1}{64j^2}.
\end{equation*}
Since $a_j+b_j\le\lambda^2+\frac{1}{32}<1$, define
\begin{equation*}\label{eq:sd-near-kernel}
Q(s_j,\cdot)
=
(1-a_j-b_j)\delta_{-\lambda}
+
a_j\delta_{-4j}
+
b_j\delta_{R_j},
\end{equation*}
where $R_j>0$ is chosen to make the noise centered:
\begin{equation}\label{eq:sd-Rj}
R_j
=
\frac{(1-a_j-b_j)\lambda+4ja_j}{b_j}
=
\frac{(1-a_j-b_j)\lambda+\frac{1}{16j}}
{\lambda^2+\frac{1}{64j^2}}.
\end{equation}
The rare large outcome associated with label $j$ is $-4j$. It is assigned probability mass $1/(64j^2)$ by the state-dependent law at $s_j$, and it moves the iterate to the right by $4\eta j$.
At all other points $x\in\cX\setminus\{s_1,\dots,s_{n-1}\}$, set $Q(x,\cdot)=\delta_0$.

We construct the process recursively as follows. Given $\cF_k$, sample $W_k$ from $Q(X_k,\cdot)$ independently of the past, set $G_k\coloneq g(X_k)$, and update $X_{k+1}$ by \eqref{eq:sgd}.
\begin{lemma}\label{lem:sd-oracle-properties}
The kernel $Q$ is centered and has uniformly bounded second moment:
\begin{equation}\label{eq:sd-moment-assumptions}
\int_{\bbR} w\,Q(x,\dif w)=0,
\qquad
\int_{\bbR} w^2\,Q(x,\dif w)\le3,
\qquad
x\in\cX.
\end{equation}
Consequently, the oracle satisfies Assumptions~\ref{ass:additive-kernel-oracle} and~\ref{ass:time-homogeneous-noise} with $\sigma^2=3$, and hence also Assumption~\ref{ass:centered_additive_l2_noise}.
Furthermore, it is not state-independent.
\end{lemma}
\begin{proof}
At all points where $Q(x,\cdot)=\delta_0$, the claims are immediate. For a far labeled state $s_j$, $j>M$, the mean is
\begin{equation*}
(1-\lambda^2)(-\lambda) + \lambda^2\frac{1-\lambda^2}{\lambda}=0,
\end{equation*}
and the second moment is
\begin{align*}
(1-\lambda^2)\lambda^2
+\lambda^2\left(\frac{1-\lambda^2}{\lambda}\right)^2
&=(1-\lambda^2)\lambda^2+(1-\lambda^2)^2\\
&=1-\lambda^2
 \le 1.
\end{align*}
For a near labeled state $s_j$, $j\le M$, the definition of $R_j$ gives
\begin{equation*}
(1-a_j-b_j)(-\lambda)+a_j(-4j)+b_jR_j = 0,
\end{equation*}
so the noise is centered. For the second moment, using \eqref{eq:sd-Rj},
\begin{align*}
\int_{\bbR} w^2\,Q(s_j,\dif w)
&=(1-a_j-b_j)\lambda^2+a_j(4j)^2+b_jR_j^2 \\
&\le \lambda^2+\frac{1}{4}+
\frac{\bigl((1-a_j-b_j)\lambda+\frac{1}{16j}\bigr)^2}
{\lambda^2+\frac{1}{64j^2}} \\
&\le \lambda^2+\frac{1}{4}+
2\frac{\lambda^2}{\lambda^2+\frac{1}{64j^2}}
+2\frac{\left(\frac{1}{16j}\right)^2}
{\lambda^2+\frac{1}{64j^2}} \\
&\le \lambda^2+\frac{1}{4}+2+\frac{1}{2}
=\lambda^2+\frac{11}{4}.
\end{align*}
Since $\lambda\le\eta\le2^{-8}$, this is smaller than $3$.

Finally, because $X_k$ is $\cF_k$-measurable and $W_k$ is sampled from $Q(X_k,\cdot)$ conditionally independently of the past, \eqref{eq:sd-moment-assumptions} gives
\begin{equation*}
\bbE[W_k\mid\cF_k]=0,
\qquad
\bbE[W_k^2\mid\cF_k]\le3.
\end{equation*}
Finally, this verifies Assumptions~\ref{ass:additive-kernel-oracle} and~\ref{ass:time-homogeneous-noise} with $\sigma^2=3$, and hence also Assumption~\ref{ass:centered_additive_l2_noise}.
The kernel is not state-independent: for example, $Q(0,\cdot)=\delta_0$, whereas $Q(s_{n-1},\cdot)$ assigns positive mass to nonzero noise values.
\end{proof}

\noindent If the process is at a labeled state $s_j$ with $j \ge 1$ and the sampled noise is $W_k=-\lambda$, then $X_{k+1}=s_{j-1}$. Therefore, along the path where these values are sampled before the rare large outcome,
\begin{equation*}
X_k=s_{n-k}=-(n-k+1)\eta \lambda,
\qquad k=1,2,\dots,n.
\end{equation*}
In particular, $s_j$ is visited at time $k_j = n-j$.

For $j=1,\dots,M$, define
\[
    E_j
\coloneq
    \left(\bigcap_{r=j+1}^{n-1}\{W_{n-r}=-\lambda\}\right)
    \cap
    \{W_{n-j}=-4j\}.
\]
On this event, a deterministic induction gives $X_{n-r}=s_r$ for every $r=j+1,\ldots,n-1$, and hence the event indeed corresponds to following the labeled path down to $s_j$ and then taking the rare large outcome at $s_j$.

\begin{lemma}\label{lem:sd-event-probability}
For every $j = 1,\dots,M$,
\begin{equation*}
\Pb(E_j)\ge\frac{1}{512j^2}.
\end{equation*}
Moreover, the events $E_1,\dots,E_M$ are pairwise disjoint.
\end{lemma}
\begin{proof}
By construction and conditional independence,
\begin{equation}\label{eq:sd-Ej-probability}
\Pb(E_j)
=
a_j
\prod_{r=j+1}^{M}(1-a_r-b_r)
\prod_{r=M+1}^{n-1}(1-\lambda^2),
\end{equation}
with empty products interpreted as one. Since $\lambda^2\le1/n$,
\begin{equation*}
\prod_{r=M+1}^{n-1}(1-\lambda^2)
\ge(1-1/n)^n
\ge\frac{1}{4}.
\end{equation*}
For the near-labeled product, applying $\prod_i(1-u_i) \ge 1-\sum_i u_i$, valid for $u_r\in[0,1]$, to $u_r \coloneq a_r+b_r=\lambda^2+\frac{1}{32r^2} \in [0,1]$ and recalling that $\sum_{r=1}^{\infty}{r^{-2}} = \pi^{2}/6$ we have
\begin{align*}
\prod_{r=j+1}^{M}(1-a_r-b_r)
&\ge
1-\sum_{r=j+1}^{M}\left(\lambda^2+\frac{1}{32r^2}\right) \\
&\ge
1-M\lambda^2-\frac{1}{32}\sum_{r=1}^{\infty}\frac{1}{r^2} \\
&\ge
1-M\eta^2-\frac{\pi^2}{192} \\
&\ge
1-\frac{\eta}{16}-\frac{\pi^2}{192}
\ge\frac{1}{2}.
\end{align*}
Plugging these bounds and $a_j=1/(64j^2)$ into \eqref{eq:sd-Ej-probability} gives
\begin{equation*}
\Pb(E_j)\ge\frac{1}{8} \cdot \frac{1}{64 j^2}=\frac{1}{512j^2}.
\end{equation*}
We observe that, for every $1 \le j<\ell \le M$, the event $E_j$ requires the noise realization $W_k=-\lambda$ at $s_{\ell}$, while $E_{\ell}$ requires the rare large outcome $W_{k_\ell}=-4\ell$ at $s_{\ell}$. Hence the events are pairwise disjoint.
\end{proof}
\noindent Now we show that for the events $E_j$ the hypotheses of Lemma~\ref{lem:clock-payoff} hold. On $E_j$, the process reaches $s_j = -(j+1)\eta \lambda$ at time $k_j = n-j$, and $W_{k_j}=-4j$. The rare large outcome sends the iterate to a point at least $4\eta j-(j+1)\eta \lambda>0$. From then on, as long as the iterate is positive, the kernel is $\delta_0$; with zero noise, each deterministic update decreases the iterate by at most $\eta$. Hence, after any $h\le j-1$ subsequent updates, the iterate is at least $4\eta j-(j+1)\eta \lambda-h\eta$, which is positive. Therefore all subsequent noises up to time $n-1$ are zero, and the hypotheses of Lemma~\ref{lem:clock-payoff} hold.

Now we have all the ingredients to complete the proof of the time-homogeneous state-dependent case (Item~\ref{it:state-dependent-lb} of Theorem~\ref{thm:clock-counterexamples}).

Lemma~\ref{lem:sd-event-probability} and Lemma~\ref{lem:clock-payoff} verify the assumptions of Lemma~\ref{lem:clock-template} with $Y=f(X_n)-f_\star$, $p=1/512$, and $\ell=2$. Hence
\begin{equation*}
\bbE[f(X_{n})-f_{\star}]
\ge
\frac{1}{512}\,\eta\log\frac1\eta.
\end{equation*}
Since $\eta=c_n/\sqrt n$, and using that  $\underline c \le c_n \le \overline c$ and \eqref{eq:theta-lb-size-conditions},
\begin{equation*}
\eta\log\frac1\eta
=
\frac{c_n}{\sqrt n}\log\frac{\sqrt n}{c_n}
\ge
\frac{\underline c}{\sqrt n}\log\frac{\sqrt n}{\overline c}
\ge
\frac{\underline c}{4}\frac{\log n}{\sqrt n}.
\end{equation*}
This proves \eqref{eq:sd-main-lower-bound}.
Finally, the support of $Q(s_M,\cdot)$ contains the value $-4M$. Since $1/(16\eta) \ge 16$, we have $M \ge 1/(32\eta)$, and therefore
\begin{equation*}
4M\ge\frac{1}{8\eta}\ge\frac{\sqrt n}{8\overline c}.
\end{equation*}
Thus, for $x=s_M$, the support of $Q(x,\cdot)$ contains a value of magnitude at least $\sqrt n/(8\overline c)$.

\subsection{A state-independent time-inhomogeneous implementation}\label{subsec:time-inhomogeneous-clock}
This subsection proves item~\ref{it:time-inhomogeneous-lb} of Theorem~\ref{thm:clock-counterexamples}. The objective and subgradient selection are the same as in \eqref{eq:clock-common-objective}--\eqref{eq:clock-common-g-selection}, but the process starts from the optimal point $X_1\equiv0$. Here the labels $j$ are attached directly to some time indices of the iterates, thereby determining when the noise can produce a rare large outcome.

For $j = 1,\ldots,M$, define the special time
\begin{equation*}
k_j = n-j.
\end{equation*}
Let the random variables $(W_k)_{k \in \bbN}$ be independent. At a special time $k_j$, set
\begin{equation*}\label{eq:ti-noise-law}
W_{k_j}=
\begin{cases}
-4j, & \text{with probability } 1/(64j^2),\\
+4j, & \text{with probability } 1/(64j^2),\\
0, & \text{otherwise},
\end{cases}
\end{equation*}
and at every non-special time set $W_k=0$. This sequence is state-independent because the law of $W_k$ depends only on $k$; in kernel notation, set $Q_k(x,\cdot)=\law(W_k)$ for every $x\in\cX$.

The rare large outcome associated with label $j$ is again $-4j$, but now it can be sampled from the law of the time-indexed noise variable $W_{k_j}$.
\begin{lemma}\label{lem:ti-oracle-properties}
The associated kernels $Q_k(x,\cdot)=\law(W_k)$ are centered and state-independent. Moreover, the variables $(W_k)_{k\in\bbN}$ are independent across time and satisfy
\begin{equation*}
\bbE[W_k^2]\le\frac{1}{2},
\qquad k \ge 1.
\end{equation*}
In particular, the oracle satisfies Assumptions~\ref{ass:additive-kernel-oracle} and~\ref{ass:state-independent-noise} with $\sigma^2=1/2$, and hence also Assumption~\ref{ass:centered_additive_l2_noise}.
Finally, it is not time-homogeneous.
\end{lemma}
\begin{proof}
At non-special times the claim is immediate. At time $k_j$, symmetry gives $\bbE[W_{k_j}]=0$, and
\begin{equation*}
\bbE[W_{k_j}^2]
=
2\frac{1}{64j^2}(4j)^2
=
\frac{1}{2}.
\end{equation*}
Since the variables are independent and their laws are fixed in advance, for every $k$ the conditional law given $\cF_k$ is $Q_k(X_k,\cdot)=\law(W_k)$, and the conditional mean and second moment are the corresponding unconditional quantities. This verifies Assumptions~\ref{ass:additive-kernel-oracle} and~\ref{ass:state-independent-noise} with $\sigma^2=1/2$, and hence also Assumption~\ref{ass:centered_additive_l2_noise}.
The oracle is not time-homogeneous: non-special times have kernel $\delta_0$, whereas each special time $k_j$ has a kernel assigning positive mass to $\pm4j$.
\end{proof}
\noindent For $j = 1,\ldots,M$, let $F_j$ be the event that $W_{k_j}=-4j$ and $W_{k_\ell}=0$ for every $\ell\neq j$ with $1\le\ell\le M$.
\begin{lemma}\label{lem:ti-event-probability}
For every $j = 1,\ldots,M$,
\begin{equation*}
\Pb(F_j)\ge\frac{1}{128j^2}.
\end{equation*}
Moreover, the events $F_1,\ldots,F_M$ are pairwise disjoint. On $F_j$, the hypotheses of Lemma~\ref{lem:clock-payoff} hold.
\end{lemma}
\begin{proof}
By independence,
\begin{equation*}
\Pb(F_j)
=
\frac{1}{64j^2}
\prod_{\ell\neq j}\left(1-\frac{1}{32\ell^2}\right),
\end{equation*}
where the product is over $\ell\in\{1,\ldots,M\}\setminus\{j\}$. Since $\frac{1}{32}\sum_{\ell \ge 1}\ell^{-2}<1/2$ and $\prod_i(1-u_i) \ge 1-\sum_i u_i$ for $u_i\in[0,1]$,
\begin{equation*}
\prod_{\ell\neq j}\left(1-\frac{1}{32\ell^2}\right)
\ge
1-\frac{1}{32}\sum_{\ell=1}^{\infty}\frac{1}{\ell^2}
\ge\frac{1}{2}.
\end{equation*}
Therefore $\Pb(F_j)\ge 1/(128j^2)$.
Pairwise disjointness follows because, if $j\neq\ell$, then $F_j$ requires $W_{k_\ell}=0$, whereas $F_\ell$ requires $W_{k_\ell}=-4\ell$.

On $F_j$, all special noises before time $k_j$ are zero, and all non-special noises are zero by construction. Since $X_1=0$ and $g(0)=0$, the process stays at $0$ until time $k_j$. At time $k_j$, $W_{k_j}=-4j$, and all later special noises are zero on $F_j$, while all later non-special noises are zero by construction. Thus the hypotheses of Lemma~\ref{lem:clock-payoff} hold with $X_{k_j}=0\in[-(j+1)\eta \lambda, 0]$.
\end{proof}

Now we have all the ingredients to complete the proof of the state-independent time-inhomogeneous case (Item~\ref{it:time-inhomogeneous-lb} of Theorem~\ref{thm:clock-counterexamples}).

Lemma~\ref{lem:ti-event-probability} and Lemma~\ref{lem:clock-payoff} verify the assumptions of Lemma~\ref{lem:clock-template} with $Y=f(X_n)-f_\star$, $p=1/128$, and $\ell=2$. Hence
\begin{equation*}
\bbE[f(X_n)-f_\star]
\ge
\frac{1}{128}\,\eta\log\frac1\eta.
\end{equation*}
Using the same conversion from $\eta\log(1/\eta)$ to $(\log n)/\sqrt n$ as in the state-dependent case, we get
\begin{equation*}
\bbE[f(X_n)-f_\star]
\ge
\frac{\underline c}{512}\frac{\log n}{\sqrt n}
\ge
\frac{\underline c}{2^{11}}\frac{\log n}{\sqrt n}.
\end{equation*}
This proves \eqref{eq:ti-main-lower-bound}.
Moreover, the support of $W_{k_M}$ contains the values $\pm4M$, and therefore contains values of magnitude at least $1/(8\eta)\ge \sqrt n/(8\overline c)$ by the same estimate as in the state-dependent construction.
\section{Final remarks}
The proof of Lemma~\ref{lem:frozen} uses the combination of Assumptions~\ref{ass:time-homogeneous-noise} and~\ref{ass:state-independent-noise} in a very concrete way. After conditioning on the good iterate $X_{k_{\star}}$, the future noises are still independent copies of the same law, regardless of where the process moves.
This lets us define the two one-sided dominating chains in \eqref{eq:branch-chain} and apply Proposition~\ref{prop:technical} or Corollary~\ref{cor:technical-extended}, depending on whether the corresponding side of the domain is finite or infinite.
All the argument is made possible by the Foster--Lyapunov-type inequality \eqref{eq:fl-type-ineq} that is derived from the fundamental inequality \eqref{eq:cornerstone-ineq}, which enables the construction of a suitable potential $B$ as described in Claim~\ref{claim:BC}.

Theorem~\ref{thm:clock-counterexamples} shows that both kernel symmetries matter under conditional second-moment control only. If state-independence is dropped, the current state can determine the availability of a rare large noise outcome. The kernel is still fixed in time and centered at every state, but different states can expose different rare large outcomes. If time-homogeneity is dropped, the time index can play the same role: the state-independent laws associated with specific time indices can sample the corresponding rare large outcome. In both cases, the logarithm term cannot be ruled out.

The examples use noise realizations of order $1/\eta$, and hence of order $\sqrt n$ for every standard fixed-stepsize family $\eta=\Theta(1/\sqrt n)$. Thus they do not settle the corresponding models with uniformly almost-surely bounded stochastic gradients. What they do settle is that the plain variance condition, Assumption~\ref{ass:centered_additive_l2_noise}, is not sufficient to remove the logarithm. Even strengthening it by requiring only time-homogeneity, or only state-independence, is still insufficient. The logarithm disappears in our proof only when both structural symmetries are imposed simultaneously. Finally,  an interesting open problem, is to establish if in the deterministic  multivariate setting (for each fixed dimension $d > 1$), the last iterate is optimal.

\section*{Acknowledgments}

TC gratefully acknowledges the support of the Natural Sciences and Engineering Research Council of Canada (NSERC) through grant RGPIN-2023-03688 (Discovery Grants Program).

\bibliographystyle{plain}
\bibliography{biblio}

@article{Agarwal2012,
author={Agarwal, Alekh and Bartlett, Peter L. and Ravikumar, Pradeep and Wainwright, Martin J.},
title={Information-Theoretic Lower Bounds on the Oracle Complexity of Stochastic Convex Optimization}, 
journal={IEEE Transactions on Information Theory}, 
year={2012},
volume={58},
number={5},
pages={3235-3249},
}

@book{Bach2024,
title={Learning theory from first principles},
author={Bach, Francis},
year={2024},
publisher={MIT press}
}

@inproceedings{Eldowa2024,
title={General Tail Bounds for Non-Smooth Stochastic Mirror Descent},
author={Eldowa, Khaled and Paudice, Andrea},
booktitle={Proceedings of the 27-th International Conference on Artificial Intelligence and Statistics},
pages={3205--3213},
year={2024}
}

@article{Ermoliev1969,
title={On the method of generalized stochastic gradients and quasi-{F}\'ejer sequences},
author={Ermol'ev, Yu. M.},
journal={Cybernetics},
Volume = {5},
pages={208--220},
year={1969}
}

@article{Liu2023b,
title={Stochastic Nonsmooth Convex Optimization with Heavy-Tailed Noises: High-Probability Bound, In-Expectation Rate and Initial Distance Adaptation},
author={Liu, Zijian and Zhou, Zhengyuan},
journal={arXiv preprint arXiv:2303.12277},
year={2023}
}

@inproceedings{Liu2024,
author={Liu, Zijian and Zhou, Zhengyuan},
title={Revisiting the last-iterate convergence of stochastic gradient methods},
booktitle={Proceedings of the 12-th International Conference on Learning Representations (to appear)},
year={2024}
}

@inproceedings{Harvey2019a,
title={Tight analyses for non-smooth stochastic gradient descent},
author={Harvey, Nicholas J.A. and Liaw, Christopher and Plan, Yaniv and Randhawa, Sikander},
booktitle={Proceedinds of the 32nd International Conference on Computational Learning Theory},
pages={1579--1613},
year={2019}
}

@article{Harvey2024,
author={Nicholas J. A. Harvey and Chris Liaw and Sikander Randhawa},
title={Tight analyses for subgradient descent {I:} Lower bounds},
journal={Open J. Math. Optim.},
volume={5},
pages={1--17},
year={2024}
}

@article{Jain2021,
title = {Making the Last Iterate of SGD Information Theoretically Optimal},
author = {Jain, Prateek and Nagaraj, Dheeraj M. and Netrapalli, Praneeth},
journal = {SIAM Journal on Optimization},
volume = {31},
number = {2},
pages = {1108-1130},
year = {2021},
}

@book{Nemirovski1983,
title={Problem complexity and method efficiency in optimization},
author={Nemirovskij, Arkadij Semenovi{\v{c}} and Yudin, David Borisovich},
year={1983},
publisher={Wiley-Interscience}
}

@article{Nemirovski2009,
author = {Nemirovski, A. and Juditsky, A. and Lan, G. and Shapiro, A.},
title = {Robust Stochastic Approximation Approach to Stochastic Programming},
journal = {SIAM Journal on Optimization},
volume = {19},
number = {4},
pages = {1574-1609},
year = {2009}
}

@misc{Parletta2025,
author={Parletta, Daniela Angela and Paudice, Andrea and Salzo, Saverio},
title={An Improved Analysis of the Clipped Stochastic subGradient Method under Heavy-Tailed Noise},
year={2025},
eprint={2410.00573},
archivePrefix={arXiv},
url={https://arxiv.org/abs/2410.00573}
}

@inproceedings{Shamir2013,
title={Stochastic Gradient Descent for Non-smooth Optimization: Convergence Results and Optimal Averaging Schemes},
author={Shamir, Ohad and Zhang, Tong},
booktitle={Proceedings of the 30th International Conference on Machine Learning},
year={2013}
}

@inproceedings{koren-segal2020,
  author    = {Tomer Koren and Shahar Segal},
  title     = {Open Problem: Tight Convergence of {SGD} in Constant Dimension},
  booktitle = {COLT 2020},
  series    = {Proceedings of Machine Learning Research},
  volume    = {125},
  pages     = {3847--3851},
  year      = {2020}
}

@misc{Kornowski2026,
author       = {Guy Kornowski and Ohad Shamir},
title        = {Gradient Descent's Last Iterate Is Often (Slightly) Suboptimal},
howpublished = {Preprint at arXiv:2604.13870},
month        = apr,
year         = {2026}
}

@misc{liu-lu2021,
  author       = {Daogao Liu and Zhou Lu},
  title        = {The Convergence Rate of {SGD}'s Final Iterate: Analysis on Dimension Dependence},
  note = {Preprint arXiv:2106.14588},
  year         = {2021}
}

@book{meyn2009,
  author    = {Sean P. Meyn and Richard L. Tweedie},
  title     = {Markov Chains and Stochastic Stability},
  publisher = {Cambridge University Press},
  edition   = {2},
  year      = {2009}
}

@article{zamani-glineur2025,
  author  = {Moslem Zamani and Fran{\c{c}}ois Glineur},
  title   = {Exact Convergence Rate of the Last Iterate in Subgradient Methods},
  journal = {SIAM Journal on Optimization},
  volume  = {35},
  number  = {3},
  pages   = {2182--2201},
  year    = {2025}
}

@misc{yuksel2012,
      title={Random-Time, State-Dependent Stochastic Drift for {M}arkov Chains and Application to Stochastic Stabilization Over Erasure Channels}, 
      author={Serdar Y\"uksel and Sean P. Meyn},
      year={2012},
      note= {Preprint at arXiv:1010.4820},
}

@book{kallenberg2021,
  author    = {Kallenberg, Olav},
  title     = {Foundations of Modern Probability},
  edition   = {3},
  series    = {Probability Theory and Stochastic Modelling},
  volume    = {99},
  publisher = {Springer},
  year      = {2021},
  doi       = {10.1007/978-3-030-61871-1}
}

\newpage
\appendix

\section{Proofs of the Deferred Claims}\label{sec:claim-proofs}
\subsection{Proof of Claim~\ref{claim:invariant}}
\begin{proof}
We first prove that $P$ satisfies the Feller property.
Recall that a transition probability kernel $P$ on the compact metric space $\cI$ is Feller if
\begin{equation*}
\varphi\in C(\cI)
\qquad\Longrightarrow\qquad
P\varphi\in C(\cI).
\end{equation*}
Let $\varphi\in C(\cI)$, and let $z_\ell\to z$ in $\cI$.
Since $T$ is continuous and the projection $\Pi_\cI$ is continuous, for every fixed $r\in\bbR$,
\begin{equation*}
\Phi(z_\ell,r)
=
\Pi_\cI\lrb{T(z_\ell) - \eta r}
\longrightarrow
\Pi_\cI\lrb{T(z) - \eta r}
=
\Phi(z,r).
\end{equation*}
By continuity of $\varphi$,
\begin{equation*}
\varphi\lrb{\Phi(z_\ell,r)}
\longrightarrow
\varphi\lrb{\Phi(z,r)}.
\end{equation*}
Moreover, $\labs{\varphi\lrb{\Phi(z_\ell,r)}}\le\norm{\varphi}_\infty$, and $\norm{\varphi}_\infty<\infty$ because $\cI$ is compact.
The dominated convergence theorem gives
\begin{equation*}
P\varphi(z_\ell)
=
\int_{\bbR}\varphi\lrb{\Phi(z_\ell,r)}\mu(\dif r)
\longrightarrow
\int_{\bbR}\varphi\lrb{\Phi(z,r)}\mu(\dif r)
=
P\varphi(z).
\end{equation*}
Thus $P\varphi\in C(\cI)$, and $P$ is Feller. Consequently, the chain admits an invariant probability measure by the Krylov--Bogoliubov averaging argument. For the sake of completeness, we provide the details below.
\noindent Fix $z_0\in \cI$ and define
\begin{equation*}
\nu_N
:=
\frac{1}{N}\sum_{t=0}^{N-1}\delta_{z_0}P^t,
\end{equation*}
where $\delta_{z_0}P^t$ denotes the law at time $t$ of the chain started from $z_0$.
Since $\cI$ is a compact metric space, the set of probability measures on $\cI$ is sequentially compact for weak convergence, i.e., every sequence of probability measures on $\cI$ has a weakly convergent subsequence.
Hence there exist a subsequence $(\nu_{N_j})_{j \ge 1}$ and a probability measure $\pi$ on $\cI$ such that
\begin{equation}\label{eq:weak-conv}
\nu_{N_j}\xrightarrow{\mathrm w}\pi,
\end{equation}
meaning that $\int_\cI \varphi \dif\nu_{N_j} \to \int_\cI \varphi \dif\pi$ for every $\varphi\in C(\cI)$.
\noindent Fix $\varphi\in C(\cI)$.
We know that $P$ is Feller, thus $P\varphi\in C(\cI)$, and by \eqref{eq:weak-conv} we get
\begin{equation*}
\int_\cI\varphi\dif\nu_{N_j} \longrightarrow \int_\cI \varphi \dif\pi,
\end{equation*}
and
\begin{equation*}
\int_\cI\varphi\dif(\nu_{N_j}P) = \int_\cI P\varphi \dif\nu_{N_j} \longrightarrow \int_\cI P\varphi \dif\pi = \int_\cI\varphi \dif(\pi P).
\end{equation*}
To show that these two limits coincide, note that 
\begin{equation*}
\begin{aligned}
\int_\cI \varphi \dif(\nu_NP) - \int_\cI \varphi \dif\nu_N
& = \frac{1}{N} \sum_{t=0}^{N-1} \left( P^{t+1} \varphi(z_0) - P^t \varphi(z_0) \right) 
 = \frac{P^N \varphi(z_0) - \varphi(z_0)}{N},
\end{aligned}
\end{equation*}
hence,
\begin{equation*}
\labs{ \int_\cI\varphi\dif(\nu_NP) - \int_\cI \varphi \dif\nu_N} \le \frac{2\norm{\varphi}_\infty}{N} \longrightarrow 0.
\end{equation*}
It follows that
\begin{equation*}
\int_\cI\varphi \dif(\pi P) = \int_\cI\varphi \dif\pi, \qquad \varphi\in C(\cI).
\end{equation*}
Continuous functions separate probability measures on compact metric spaces, thus $\pi P=\pi$.
\end{proof}

\subsection{Proof of Claim~\ref{claim:BC}}
\begin{proof}
Recall that $T(z) = T_{\cI, m, \eta } (z) = \sup_{ 0 \le s \le z} ( s - \eta \cdot m(s))_{+}$ for every $z \in \cI = [0, D]$, where $m: \cI \to [0,L]$ is a non-decreasing function.
For every $z \in \cI$ we set
\begin{equation*}\label{eq:B}
B(z)
\coloneq
\frac{z^2}{2 \eta} + 2 L z
\end{equation*}
and we want to show that
\begin{equation}\label{eq:fl-in-claim-proof}
\bbE[B \lrb{ \Pi_\cI \lrb{ T(z) - \eta \cdot R}}] - B(z) \le -\int_{0}^{z} m(s) \dif s + \lrb{\frac{\sigma^2}{2} + L \sigma}\cdot \eta,
\qquad z \in \cI,
\end{equation}
where the random variable $R$ satisfies $\bbE[R] = 0$ and $\bbE[R^2]\le \sigma^2$. 
Our proof of \eqref{eq:fl-in-claim-proof} will require the two following inequalities:
\begin{equation}\label{eq:B-ineq}
B (T(z)) - B(z) \le -\int_{0}^{z} m(s) \dif s \qquad \text{for every $z \in \cI$}
\end{equation}
\begin{equation}\label{eq:noise-ineq}
\bbE[ B ( \Pi_{\cI} (s - \eta \cdot R)) - B(s)]
\le
\eta \cdot \lrb{\frac{\sigma^2}{2} + L \sigma}
\qquad
\text{for every $s \in \cI$.}
\end{equation}
\noindent Assuming \eqref{eq:B-ineq} and \eqref{eq:noise-ineq}, \eqref{eq:fl-in-claim-proof} readily follows by plugging $s = T(z)$ into \eqref{eq:noise-ineq} and adding the resulting inequality to \eqref{eq:B-ineq}. It remains to prove \eqref{eq:B-ineq} and \eqref{eq:noise-ineq}.
\noindent To show \eqref{eq:noise-ineq}, let $s \in \cI$. By definition of $B$, we may rewrite the left-hand side of \eqref{eq:noise-ineq} as:
\begin{equation}\label{eq:noise-lhs-expanded}
\bbE \lsb{\frac{1}{2\eta} \cdot \lrb{(\Pi_{\cI} (s - \eta \cdot R))^2 - s^2} + 2L (\Pi_{\cI} (s - \eta \cdot R) - s ) }
\end{equation}
Note that the projection onto $[0,D]$ cannot increase the distance from $0$. In particular:
\begin{equation}\label{eq:quadratic-part}
(\Pi_{\cI} (s - \eta \cdot R))^2 - s^2 \le (s - \eta \cdot R)^2 - s^2 = \eta^2 R ^ 2 - 2 s R \eta .
\end{equation}
Also, note that $\Pi_\cI( s - \delta) \le s$ if $\delta$ is positive, whereas $\Pi_\cI( s - \delta) \le s - \delta$ if $\delta$ is non-positive, hence setting $R_{-}\coloneq \max\{ -R, 0 \}$ yields
\begin{equation}\label{eq:linear-part}
\Pi_{\cI} (s - \eta \cdot R) - s \le \eta \cdot R_{-} .
\end{equation}
Combining \eqref{eq:quadratic-part} and \eqref{eq:linear-part} with \eqref{eq:noise-lhs-expanded} gives:
\begin{align*}
\bbE[ B ( \Pi_{\cI} (s - \eta \cdot R)) - B(s)]
&\le \bbE \lsb{ \frac{1}{2\eta}(\eta^2 R ^ 2 - 2 s R \eta) + 2L \eta \cdot R_{-}} \\
&= \bbE \lsb{ \frac{1}{2} R ^ 2 + 2L R_{-}}\cdot \eta,
\end{align*}
where we used $\bbE[R]= 0$. Moreover, $\bbE[R]= 0$ entails that $\bbE[R_{-}]= \bbE[R_{+}] = \frac{1}{2} \bbE[\labs{R}]$. This gives 
\begin{equation*} 
 \bbE[ B ( \Pi_{\cI} (s - \eta \cdot R)) - B(s)] \le \lrb{\frac{1}{2} \sigma^2 + L \bbE[\labs{R}] } \cdot \eta,
\end{equation*}
and Jensen's inequality $\bbE[|R|] \le \sqrt{\bbE[R^2]} \le \sigma$ yields \eqref{eq:noise-ineq}.

\noindent To show \eqref{eq:B-ineq}, let $z \in \cI$.
By \eqref{eq:cornerstone-ineq}, we have
\begin{equation*}
\eta\int_{0}^z m(s)\dif s
\le
(z-T(z))\lrb{z+\eta L},
\end{equation*}
On the other hand, the identity $(z^2 - T(z)^2) = (z - T(z))\cdot( z + T(z))$ gives
\begin{equation*}
\eta\big(B(z)-B(T(z))\big)
=(z-T(z))\lrb{2L\eta+\frac{z+T(z)}2}.
\end{equation*}
Therefore,
\begin{align}\label{eq:almost-there}
\eta\big(B(z)-B(T(z))\big)-\eta\int_0^z m(s)\dif s
&\ge
(z-T(z))\lrb{2L\eta+\frac{z+T(z)}2-(z+\eta L)}\\
&=(z-T(z))\lrb{L\eta-\frac{z-T(z)}2}.\notag
\end{align}
Recalling that $T(z) \le z$ and noting that
\begin{equation*}
 z - T(z) = z - \sup_{0 \le s \le z} (s - \eta \cdot m(s))_{+} 
 \le z - \sup_{0 \le s \le z} (s - \eta \cdot m(s))
 \le z - \sup_{0 \le s \le z} (s - \eta \cdot L)
 = \eta \cdot L,
\end{equation*}
it follows that the right-hand side of \eqref{eq:almost-there} is non-negative, which is equivalent to \eqref{eq:B-ineq}.
\end{proof}

\subsection{Proof of Claim~\ref{claim:sandwich}}

To prove Claim~\ref{claim:sandwich}, we will use the following technical lemma.
\begin{lemma}
\label{lem:subgradient-signs}
Let $\cX\subseteq\bbR$ be a closed interval, let $f\colon\cX\to\bbR$ be convex, and let
\[
S=\{x\in\cX:f(x)\le q\}
\]
be non-empty and closed. Fix $\rho \ge 0$, and define
\[
E=\{x\in\cX:\operatorname{dist}(x,S)\le \rho\},
\qquad
\alpha=\inf E,
\qquad
\beta=\sup E,
\]
with $a=\inf\cX$ and $b=\sup\cX$.
Then the following hold.
\begin{enumerate}
\item If $\alpha>a$, then every $x\in[\alpha,\alpha+\rho)\cap\cX$ lies strictly to the left of $S$, and every relative subgradient $g\in\partial_{\cX}f(x)$ satisfies $g \le 0$.
\item If $\beta<b$, then every $x\in(\beta-\rho,\beta]\cap\cX$ lies strictly to the right of $S$, and every relative subgradient $g\in\partial_{\cX}f(x)$ satisfies $g \ge 0$.
\end{enumerate}
\end{lemma}
\begin{proof}
Since $S$ is a non-empty closed convex subset of the interval $\cX$, it is a closed interval in $\cX$. Write
\[
u=\inf S,
\qquad
v=\sup S.
\]
If $\alpha>a$, then the left endpoint of the enlarged set is not created by the boundary of $\cX$, hence $\alpha=u-\rho$. Therefore every $x\in[\alpha,\alpha+\rho)\cap\cX$ satisfies $x<u$, and so $x\notin S$. Pick any $y\in S$. Then $y>x$ and $f(y)\le q<f(x)$. For any $g\in\partial_{\cX}f(x)$, the relative subgradient inequality gives
\[
f(y)\ge f(x)+g(y-x).
\]
Since $y-x>0$, this implies
\[
g\le \frac{f(y)-f(x)}{y-x}<0.
\]
In particular $g \le 0$.

The right side is symmetric. If $\beta<b$, then $\beta=v+\rho$. Hence every $x\in(\beta-\rho,\beta]\cap\cX$ satisfies $x>v$, and so $x \notin S$. Pick any $y\in S$. Then $y<x$ and $f(y)\le q<f(x)$. For $g\in\partial_{\cX}f(x)$,
\[
f(y)\ge f(x)+g(y-x).
\]
Since $y-x<0$, we obtain
\[
g \ge \frac{f(y)-f(x)}{y-x}>0.
\]
In particular $g \ge 0$.
\end{proof}

\begin{proof}[Proof of Claim~\ref{claim:sandwich}]
We prove the claim for each present branch separately.

\noindent Consider first the left branch. The base step $k = k_{\star}$ holds because $X_{k_{\star}}\in S_q\subseteq E_q$, hence $X_{k_{\star}} \ge \alpha$ and $z^{(1)}(X_{k_{\star}})=0=Z_{k_{\star}}^{(1)}$. Assume now that $z^{(1)}(X_k)\le Z_k^{(1)}$ for some $k\ge k_{\star}$. We first show that
\begin{equation}\label{eq:denoised-1}
X_k-\eta G_k\ge \alpha-T^{(1)}(Z_k^{(1)}).
\end{equation}
If $z^{(1)}(X_k)>0$, then $X_k<\alpha$. Since $G_k\le D_{+}f(X_k)$ and $D_{+}f(\alpha-z)=-D_{-}h^{(1)}(z)$, we get
\begin{align*}
X_k-\eta G_k
&\ge X_k-\eta D_{+}f(X_k)\\
&=\alpha-z^{(1)}(X_k) + \eta m^{(1)}(z^{(1)}(X_k))\\
&=\alpha-\lrb{z^{(1)}(X_k) - \eta m^{(1)}(z^{(1)}(X_k))}\\
&\ge \alpha-T^{(1)}(z^{(1)}(X_k))
\ge \alpha-T^{(1)}(Z_k^{(1)}),
\end{align*}
where the last step uses the monotonicity of $T^{(1)}$ and the induction hypothesis.
If instead $z^{(1)}(X_k)=0$, then $X_k\ge\alpha$. If $G_k \le 0$, then $X_k-\eta G_k\ge X_k\ge\alpha$. If $G_k>0$, then $X_k$ cannot belong to $[\alpha,\alpha+\eta L)\cap\cX$: by Lemma~\ref{lem:subgradient-signs}, applied with $S=S_q$ and $\rho=\eta L$, every relative subgradient of $f$ at any point of this interval is non-positive, contradicting $G_k\in\partial_{\cX}f(X_k)$ and $G_k>0$. Hence $X_k\ge\alpha+\eta L$, and therefore $X_k-\eta G_k\ge\alpha+\eta(L-G_k)\ge\alpha$. Since $T^{(1)} \ge 0$, \eqref{eq:denoised-1} follows.
Using \eqref{eq:denoised-1},
\begin{equation*}
X_k-\eta(G_k+W_k)
\ge
\alpha-\lrb{T^{(1)}(Z_k^{(1)}) + \eta W_k}.
\end{equation*}
Let $a=\inf\cX$ and recall that $D^{(1)}=\alpha-a\in(0,+\infty]$. By monotonicity of one-dimensional projections,
\begin{align*}
X_{k+1}
&=\Pi_\cX\lrb{X_k-\eta(G_k+W_k)}\\
&\ge \Pi_{[a,\alpha]}\lrb{X_k-\eta(G_k+W_k)}\\
&\ge \Pi_{[a,\alpha]}\lrb{\alpha-\lrb{T^{(1)}(Z_k^{(1)}) + \eta W_k}}\\
&=\alpha-\Pi_{[0,D^{(1)}]}\lrb{T^{(1)}(Z_k^{(1)}) + \eta W_k}\\
&=\alpha-Z_{k+1}^{(1)}.
\end{align*}
This is exactly $z^{(1)}(X_{k+1})\le Z_{k+1}^{(1)}$.

\noindent The right branch is symmetric. The base step follows from $X_{k_{\star}}\in E_q$, so $z^{(2)}(X_{k_{\star}})=0=Z_{k_{\star}}^{(2)}$. Assume $z^{(2)}(X_k)\le Z_k^{(2)}$. We first prove
\begin{equation}\label{eq:denoised-2}
X_k-\eta G_k\le \beta+T^{(2)}(Z_k^{(2)}).
\end{equation}
If $z^{(2)}(X_k)>0$, then $X_k>\beta$. Since $G_k\ge D_{-}f(X_k)$ and $D_{-}f(\beta+z)=D_{-}h^{(2)}(z)$, we have
\begin{align*}
X_k-\eta G_k
&\le X_k-\eta D_{-}f(X_k)\\
&=\beta+z^{(2)}(X_k) - \eta m^{(2)}(z^{(2)}(X_k))\\
&\le \beta+T^{(2)}(z^{(2)}(X_k))
\le \beta+T^{(2)}(Z_k^{(2)}).
\end{align*}
If $z^{(2)}(X_k)=0$, then $X_k\le\beta$. If $G_k\ge0$, then $X_k-\eta G_k\le X_k\le\beta$. If $G_k<0$, then $X_k$ cannot belong to $(\beta-\eta L,\beta]\cap\cX$: by Lemma~\ref{lem:subgradient-signs}, applied with $S=S_q$ and $\rho=\eta L$, every relative subgradient of $f$ at any point of this interval is non-negative, contradicting $G_k\in\partial_{\cX}f(X_k)$ and $G_k<0$. Hence $X_k\le\beta-\eta L$, and therefore $X_k-\eta G_k\le\beta-\eta L+\eta L=\beta$. Since $T^{(2)} \ge 0$, \eqref{eq:denoised-2} follows.

Using \eqref{eq:denoised-2},
\begin{align*}
X_{k+1}
&=\Pi_\cX\lrb{X_k-\eta(G_k+W_k)}\\
&\le \Pi_{[\beta,b]}\lrb{X_k-\eta(G_k+W_k)}\\
&\le \Pi_{[\beta,b]}\lrb{\beta+T^{(2)}(Z_k^{(2)}) - \eta W_k}\\
&=\beta+\Pi_{[0,D^{(2)}]}\lrb{T^{(2)}(Z_k^{(2)}) - \eta W_k}\\
&=\beta+Z_{k+1}^{(2)},
\end{align*}
where $b=\sup\cX$ and $D^{(2)}=b-\beta\in(0,+\infty]$. This proves $z^{(2)}(X_{k+1})\le Z_{k+1}^{(2)}$ and completes the induction.
\end{proof}

\section{Proof of Claim~\ref{claim:freezing}}
\begin{proof}
For $k=k_{\star}$ the claim is immediate. Hence fix $k>k_{\star}$ and set
$h\coloneq k-k_{\star}$.

By Assumptions~\ref{ass:additive-kernel-oracle},
\ref{ass:time-homogeneous-noise}, and~\ref{ass:state-independent-noise},
there exists a centered probability law $\mu$ on $\bbR$ with second moment at most
$\sigma^2$ such that
\[
\Pb(W_m\in A\mid\cF_m)=\mu(A)
\qquad
\forall m\ge1,\ \forall A\in\borel(\bbR).
\]
A standard induction using the tower property shows that, for every $r\ge1$,
the block
\[
(W_{k_{\star}},W_{k_{\star}+1},\ldots,W_{k_{\star}+r-1})
\]
is independent of $\cF_{k_{\star}}$ and has law $\mu^{\otimes r}$.
Equivalently, the infinite future-noise sequence
\[
U_{\infty}\coloneq (W_{k_{\star}},W_{k_{\star}+1},\ldots)
\]
is independent of $\cF_{k_{\star}}$ and has law $\mu^{\otimes\bbN}$.
In particular,
\[
U_h\coloneq (W_{k_{\star}},W_{k_{\star}+1},\ldots,W_{k-1})
\]
is independent of $\cF_{k_{\star}}$ and has law $\mu^{\otimes h}$.

Let
\[
V\coloneq (X_1,W_1,\ldots,W_{k_{\star}-1}),
\]
with the obvious convention when $k_{\star}=1$. Since
\[
\cF_{k_{\star}}
=
\sigma(X_1,W_1,\ldots,W_{k_{\star}-1}),
\]
we have $\sigma(V)=\cF_{k_{\star}}$.

For every $m\ge1$, set
\[
Y_m\coloneq (X_1,W_1,\ldots,W_{m-1}),
\qquad
E_m\coloneq \cX\times\bbR^{m-1},
\]
with the convention $E_1=\cX$ and $Y_1=X_1$.
By Definition~\ref{def:sgd}, $G_m$ is $\cF_m$-measurable, and by the definition
of the filtration,
\[
\cF_m=\sigma(Y_m).
\]
Hence $\sigma(G_m)\subseteq\sigma(Y_m)$, that is, $G_m$ is $Y_m$-measurable.
Since $\bbR$ is a Borel space, the functional representation lemma of Doob
\cite[Lemma~1.14]{kallenberg2021}, applied with $f=G_m$ and $g=Y_m$, gives a
$\borel(E_m)/\borel(\bbR)$-measurable map
$\widetilde\gamma_m:E_m\to\bbR$ such that
\[
G_m=\widetilde\gamma_m(Y_m)
\qquad\text{a.s.}
\]
Replacing $\widetilde\gamma_m$ by
\[
\gamma_m\coloneq \Pi_{[-L,L]}\circ\widetilde\gamma_m,
\]
and using $G_m\in[-L,L]$ a.s., we may assume that
$\gamma_m:E_m\to[-L,L]$ and
\[
G_m=\gamma_m(X_1,W_1,\ldots,W_{m-1})
\qquad\text{a.s.}
\]
Taking a countable intersection over $m\ge1$, these identities hold
simultaneously on an event of probability one. On the same event, the
validity assumption in Definition~\ref{def:sgd} gives
\[
\gamma_m(X_1,W_1,\ldots,W_{m-1})
\in \partial_{\cX}f(X_m)\cap[-L,L]
\qquad\text{for every }m\ge1.
\]

We now define the frozen recursion. Fix
\[
v=(x_1,w_1,\ldots,w_{k_{\star}-1})\in E_{k_{\star}}
\]
and an infinite future-noise sequence
\[
u=(u_1,u_2,\ldots)\in\bbR^{\bbN}.
\]
Define
\[
\omega_m(v,u)\coloneq
\begin{cases}
w_m, & 1\le m\le k_{\star}-1,\\
u_{m-k_{\star}+1}, & m\ge k_{\star}.
\end{cases}
\]
Set $x_1(v,u)\coloneq x_1$ and, recursively for $m\ge1$,
\[
y_m(v,u)\coloneq
\bigl(x_1,\omega_1(v,u),\ldots,\omega_{m-1}(v,u)\bigr)\in E_m,
\]
\[
g_m(v,u)\coloneq \gamma_m\bigl(y_m(v,u)\bigr),
\]
and
\[
x_{m+1}(v,u)
\coloneq
\Pi_{\cX}\Bigl(x_m(v,u)-\eta\bigl(g_m(v,u)+\omega_m(v,u)\bigr)\Bigr).
\]
The map $(v,u)\mapsto x_m(v,u)$ is Borel for every $m$.
Moreover, $x_{k_{\star}}(v,u)$ depends only on $v$, so we write it as
$x_{k_{\star}}(v)$.

Define the non-negative Borel map $\Psi_h:E_{k_{\star}}\times\bbR^h\to\bbR_+$
by
\[
\Psi_h(v,u_1,\ldots,u_h)
\coloneq
\bigl(f(x_k(v,u))-f(x_{k_{\star}}(v))\bigr)_+,
\]
where the right-hand side uses only the first $h$ future-noise coordinates.
By construction and by the simultaneous almost-sure identities above,
\[
\Psi_h(V,U_h)
=
\bigl(f(X_k)-f(X_{k_{\star}})\bigr)_+
\qquad\text{a.s.}
\]
Since $U_h$ is independent of $\cF_{k_{\star}}$, has law $\mu^{\otimes h}$,
and $V$ is $\cF_{k_{\star}}$-measurable, the standard conditional-independence
identity gives
\[
\bbE\left[
\bigl(f(X_k)-f(X_{k_{\star}})\bigr)_+
\mid \cF_{k_{\star}}
\right]
=
\left[
\int_{\bbR^h}
\Psi_h(v,u)\,\mu^{\otimes h}(\dif u)
\right]_{v=V}
\qquad\text{a.s.}
\]

It remains to bound the frozen integral for $\law(V)$-a.e. frozen past $v$.
For each $m\ge k_{\star}$, let $A_m\subseteq E_{k_{\star}}\times\bbR^{\bbN}$
be the measurable set of pairs $(v,u)$ such that
\[
g_j(v,u)\in\partial_{\cX}f(x_j(v,u))\cap[-L,L]
\qquad
\text{for every }1\le j\le m.
\]
Under the original law, $(V,U_{\infty})\in A_m$ almost surely. Since
$U_{\infty}$ is independent of $V$ and has law $\mu^{\otimes\bbN}$, this means
\[
\bigl(\law(V)\otimes\mu^{\otimes\bbN}\bigr)(A_m)=1.
\]
By Fubini's theorem, there exists a set $N_m\subseteq E_{k_{\star}}$ with
$\law(V)(N_m)=0$ such that
\[
\mu^{\otimes\bbN}\bigl(A_m(v,\cdot)\bigr)=1
\qquad
\text{for every }v\notin N_m.
\]
Set
\[
N\coloneq \bigcup_{m\ge k_{\star}}N_m.
\]
Then $\law(V)(N)=0$, and for every $v\notin N$ the frozen recursion has valid
subgradient selections at all times, $\mu^{\otimes\bbN}$-almost surely.

Fix such a $v\notin N$. Let $U=(U_1,U_2,\ldots)$ have law
$\mu^{\otimes\bbN}$, and define the shifted post-$k_{\star}$ process
\[
\bar X_i\coloneq x_{k_{\star}+i}(v,U),
\qquad
\bar G_i\coloneq g_{k_{\star}+i}(v,U),
\qquad
\bar W_i\coloneq U_{i+1},
\qquad i\ge0.
\]
Then $\bar X_0=x_{k_{\star}}(v)$ is deterministic, the noises
$(\bar W_i)_{i\ge0}$ are i.i.d.\ with law $\mu$, are centered, and have second
moment at most $\sigma^2$. Moreover, by the choice of $v\notin N$,
\[
\bar G_i\in\partial_{\cX}f(\bar X_i)\cap[-L,L]
\qquad
\text{for every }i\ge0
\]
almost surely. Thus the shifted frozen process satisfies the assumptions of
Lemma~\ref{lem:frozen}. Applying that lemma with deterministic level
\[
q(v)\coloneq f(x_{k_{\star}}(v))=f(\bar X_0)
\]
gives
\[
\int_{\bbR^h}
\Psi_h(v,u)\,\mu^{\otimes h}(\dif u)
=
\bbE\bigl[(f(\bar X_h)-q(v))_+\bigr]
\le
(L+\sigma)^2\eta.
\]
Since this holds for every $v\notin N$ and $\law(V)(N)=0$, substituting
$v=V$ in the conditional-expectation identity yields
\[
\bbE\left[
\bigl(f(X_k)-f(X_{k_{\star}})\bigr)_+
\mid \cF_{k_{\star}}
\right]
\le
(L+\sigma)^2\eta
\qquad\text{a.s.}
\]
This is exactly \eqref{eq:post-good-iterate-deterioration}.
\end{proof}

\section{On General State Space Markov Chains}\label{sec:markov}
This section recalls some standard notation and known facts pertaining to general state space Markov-chains. See, e.g., \cite{meyn2009} for a systematic treatment.

\noindent Let $\cS$ be a metric space with Borel $\sigma$-field $\borel(\cS)$. A transition probability kernel is a map $P:\cS\times\borel(\cS)\to[0,1]$ such that $P(z,\cdot)$ is a probability measure for every $z\in\cS$, and $P(\cdot,\cA)$ is measurable for every $\cA\in\borel(\cS)$. If $(Z_k)_{k \in \bbN}$ is a time-homogeneous Markov chain with kernel $P$, then
\begin{equation*}
\Pb(Z_{k+1}\in\cA\mid Z_1,\dots,Z_k)=P(Z_k,\cA).
\end{equation*}
For every bounded measurable function $\varphi:\cS\to\bbR$ and every probability measure $\lambda$ on $\cS$, it is customary to write
\begin{equation*}
P\varphi(z)=\int_\cS\varphi(y)P(z,\dif y),
\qquad
(\lambda P)(\cA)=\int_\cS P(z,\cA)\lambda(\dif z).
\end{equation*}
Thus $\lambda P$ is the distribution after one step when the initial distribution is $\lambda$, and $(\lambda P)(\varphi)=\lambda(P\varphi)$.
In this way, the notation extends in a natural way the usual notation for countable state-space Markov chains, so that the role of the transition matrix is replaced by $P$.

\noindent A standard source of transition kernels on general state spaces is a stochastic recursion driven by fresh random inputs that are independent of the current state and of the past. For example, suppose that $(R_k)_{k\geq 1}$ is an i.i.d. sequence with common law $\mu$ on a measurable space $\cR$, and that $\Phi:\cS\times\cR\to\cS$ is measurable. The recursion
\begin{equation*}
Z_{k+1}=\Phi(Z_k,R_k),\qquad k\geq 1,
\end{equation*}
defines a time-homogeneous Markov chain, with transition kernel
\begin{equation*}
P(z,\cA)
=
\Pb(\Phi(z,R)\in\cA)
\end{equation*}
where $R\sim\mu$. This construction is the general-state-space analogue of specifying the one-step transition probabilities of a countable-state chain.

\noindent Of main importance in the study of Markov chains are the invariant probability measures.
A probability measure $\pi$ on $\cS$ is invariant for $P$ if
\begin{equation*}
\pi P=\pi.
\end{equation*}
Equivalently, for every bounded measurable $\varphi:\cS\to\bbR$,
\begin{equation*}
\int_\cS P\varphi(z)\pi(\dif z)
=
\int_\cS \varphi(z)\pi(\dif z).
\end{equation*}
Thus, if $Z_1\sim\pi$, then $Z_k\sim\pi$ for every $k\geq 1$. Invariant probability measures are the general-state-space version of stationary distributions.

\noindent One common route to the existence of invariant measures uses compactness and continuity. A kernel $P$ on a compact metric space $\cS$ is called Feller if
\begin{equation*}
\varphi\in C(\cS)
\qquad\Longrightarrow\qquad
P\varphi\in C(\cS).
\end{equation*}
\noindent This is often proved via the Krylov--Bogoliubov averaging argument, which guarantees that every Feller kernel on a compact metric space admits at least one invariant probability measure. Indeed, starting from a point $z_0\in\cS$, one considers the empirical averages of the laws
\begin{equation*}
\nu_N
:=
\frac1N\sum_{k=0}^{N-1}\delta_{z_0}P^k .
\end{equation*}
and shows that a subsequence converges weakly to an invariant probability measure. This compact Feller argument is only one way to prove existence of invariant measures; the broader theory also provides recurrence and tightness criteria for noncompact spaces.

\noindent A central tool in general-state-space Markov-chain theory is the Foster--Lyapunov drift inequality. For a non-negative measurable function $V:\cS\to[0,\infty]$, define the one-step drift by $PV(z)-V(z)$
whenever the terms are well defined. A typical Foster--Lyapunov inequality has the form
\begin{equation}\label{eq:general-foster-lyapunov}
PV(z)-V(z)
\leq
-f(z)+b\mathbf 1_C(z),
\qquad z\in\cS,
\end{equation}
where $f:\cS\to[0,\infty)$ is a measurable function, $C\in\borel(\cS)$, and $b<\infty$. Here the function $V$ is called the Lyapunov function, and it is chosen so that its expected one-step variation is negative away from a controlled region, whereas the function $f$ is the quantity controlled by the drift, the set $C$ is the exceptional region where the negative drift may fail, and the constant $b$ measures the size of this local compensation.

\noindent Inequalities of the form \eqref{eq:general-foster-lyapunov} play a role in the proof of several useful properties of a chain, such as recurrence, positive recurrence, existence of invariant measures, integrability of invariant measures, moment bounds, tail bounds, and convergence in weighted norms. The common mechanism is that a negative expected drift of $V$ prevents the chain from spending too much time in regions where $f$ is large.

\noindent The simplest consequence appears after integrating a drift inequality against an invariant probability measure (see,  e.g., \cite[Theorem 2.2]{yuksel2012}). If $\pi P=\pi$ and the integrals are justified, then \eqref{eq:general-foster-lyapunov} gives
\begin{equation*}
\int_\cS f(z)\pi(\dif z)
\leq
b\,\pi(C)
\leq b.
\end{equation*}
Indeed, invariance cancels the two Lyapunov terms:
\begin{equation*}
\int_\cS PV(z)\pi(\dif z)
=
\int_\cS V(z)\pi(\dif z).
\end{equation*}
\noindent  When $V$ is bounded, as happens automatically on a compact state space if $V$ is continuous, this cancellation is immediate. In more general settings, one often obtains the same conclusion through standard truncation or comparison arguments.

\noindent  This stationary expectation bound can be used as a form of moment and tail control. Indeed, if $f$ grows at infinity, then $\pi(f)<\infty$ says that the invariant distribution has limited mass in the tails. For example, choosing $f(z)$ comparable to $|z|^p$ gives a $p$-th moment bound, while choosing $f(z)$ comparable to $\exp(\theta |z|)$ gives exponential integrability. Such integrability estimates can then be converted into tail estimates by elementary inequalities.
\end{document}